\documentclass[11pt]{article}
\usepackage[usenames]{color} 
\usepackage[tbtags]{amsmath}
\usepackage{amssymb,amsthm,amsfonts} 
\usepackage[utf8]{inputenc} 
\usepackage{hyperref, bm, bbm}
\usepackage{tikz}
\usepackage{listings}
\usepackage{booktabs}
\usepackage{caption}
\usepackage{colonequals}
\usepackage[bottom=0.9in,top=1in,left=1in,right=1in]{geometry}
\usepackage{mathtools}
\usepackage{graphicx}
\usepackage{multicol}
\usepackage{multirow}
\usepackage{accents}
\usepackage{caption}
\usepackage{subfig}
\usepackage{mathrsfs, changepage}
\usepackage{enumitem,kantlipsum, lipsum}
\usepackage{tikz}
\usetikzlibrary{shapes.geometric}
\usepackage{siunitx}
\usetikzlibrary{calc}

\usepackage{enumitem}

\usepackage{fancyhdr}
\usepackage{lastpage}
\usetikzlibrary{arrows,arrows.meta, shapes.geometric}
\usetikzlibrary{shadows}
\usetikzlibrary{shadows.blur}

\usepackage{xcolor}

\DeclareRobustCommand\dout{\bgroup
 \markoverwith{\lower-0.6ex\hbox
 {\kern-0.02em\vbox{\hrule width.2em\kern0.35ex\hrule}\kern-.06em}}%
 \ULon}

\makeatletter
\newcommand{\distas}[1]{\mathbin{\overset{#1}{\kern\z@\sim}}}%
\newsavebox{\mybox}\newsavebox{\mysim}
\newcommand{\distras}[1]{%
  \savebox{\mybox}{\hbox{\kern3pt$\scriptstyle#1$\kern3pt}}%
  \savebox{\mysim}{\hbox{$\sim$}}%
  \mathbin{\overset{#1}{\kern\z@\resizebox{\wd\mybox}{\ht\mysim}{$\sim$}}}%
}

\makeatother
\makeatletter
\newtheoremstyle{problemstyle}  							
        {3pt}                                               
        {3pt}                                               
        {\normalfont}                               		
        {}                                                  
        {\bfseries}                 						
        {\normalfont\bfseries.}         					
        {.5em}                                          	
        {}       

\theoremstyle{plain}
\newtheorem*{remark}{Remark}
\newcommand*{\rom}[1]{\expandafter\@slowromancap\romannumeral #1@}

\makeatletter
\newcommand{\mathleft}{\@fleqntrue\@mathmargin0pt}
\newcommand{\mathcenter}{\@fleqnfalse}
\makeatother

\newtheorem{theorem}{Theorem}[section]

\newtheorem{definition}[theorem]{Definition}
\newtheorem{proposition}[theorem]{Proposition}

\newtheorem{corollary}[theorem]{Corollary}
\newtheorem{example}[theorem]{Example}
\newtheorem{assumption}{Assumption}

\newtheorem{lemma}[theorem]{Lemma}

\def\E{\mathbb{E}}
\def\Ex{\mathop{\mathbb{E}}}
\def\Cov{\mathrm{Cov}}
\def\Covx{\mathop{\mathrm{Cov}}}

\def\Varx{\mathop{\mathrm{Var}}}

\def\R{\mathbb{R}}

\def\Q{\mathbb{Q}}
\def\P{\mathbb{P}}
\def\N{\mathcal{N}}
\def\E{\mathbb{E}}

\def\sN{\mathcal{N}}
\def\sP{\mathcal{P}}
\def\sX{\mathcal{X}}
\def\sI{\mathcal{I}}

\def\sG{\mathcal{G}}
\def\sH{\mathcal{H}}

\usepackage[mathscr]{euscript}
\def\fJ{\mathscr{J}}

\def\dir{\mathrm{Direct}}
\def\LinAMP{\mathrm{LinAMP}}
\def\wh{\mathrm{White}}

\def\pto{\xrightarrow{\mathbb{P}}}
\def\dto{\xrightarrow{\mathrm{(d)}}}

\renewcommand{\epsilon}{\varepsilon}

\renewcommand{\vec}{\mathrm{vec}}

\def\sym{\mathrm{sym}}
\def\what{\widehat}

\def\diag{\mathrm{diag}}
\def\Diag{\mathrm{diag}}

\def\Tr{\mathrm{Tr}}
\def\rank{\mathrm{rank}}
\def\cyc{\mathrm{cyc}}
\def\conn{\mathrm{conn}}
\def\Part{\mathrm{Part}}
\def\EvenPart{\mathrm{EvenPart}}
\def\Match{\mathrm{Match}}
\def\Unif{\mathrm{Unif}}
\def\Pois{\mathrm{Pois}}

\definecolor{cblack}{rgb}{0,0,0}
\definecolor{cblue}{rgb}{0.121569,0.466667,0.705882}    
\definecolor{corange}{rgb}{1.000000,0.498039,0.054902}  
\definecolor{cgreen}{rgb}{0.172549,0.627451,0.172549}   
\definecolor{cred}{rgb}{0.839216,0.152941,0.156863}     
\definecolor{cpurple}{rgb}{0.580392,0.403922,0.741176}  
\definecolor{cbrown}{rgb}{0.549020,0.337255,0.294118}   
\definecolor{cpink}{rgb}{0.890196,0.466667,0.760784}
\definecolor{cgray}{rgb}{0.498039,0.498039,0.498039}
\definecolor{cgreen2}{rgb}{0.7372549019607844, 0.7411764705882353, 0.13333333333333333}

\renewcommand{\emptyset}{\varnothing}

\hypersetup{
  linkcolor  = cblue,
  citecolor  = cgreen,
  urlcolor   = corange,
  colorlinks = true,
}

\newcommand\numberthis{\addtocounter{equation}{1}\tag{\theequation}}

\title{Computational and statistical lower bounds for low-rank estimation under general inhomogeneous noise}

\date{October 9, 2025}

\usepackage{authblk}
\author{Debsurya De\thanks{Email: \texttt{dde4@jhu.edu}}\,\,}
\author{Dmitriy Kunisky\thanks{Email: \texttt{kunisky@jhu.edu}.}}
\affil{Department of Applied Mathematics \& Statistics, Johns Hopkins University}

\begin{document}

\maketitle

\begin{abstract}
    Recent work has generalized several results concerning the well-understood spiked Wigner matrix model of a low-rank signal matrix corrupted by additive i.i.d.\ Gaussian noise to the \emph{inhomogeneous} case, where the noise has a variance profile.
    In particular, for the special case where the variance profile has a block structure, a series of results identified an effective spectral algorithm for detecting and estimating the signal, identified the threshold signal strength required for that algorithm to succeed, and proved information-theoretic lower bounds that, for some special signal distributions, match the above threshold.
    We complement these results by studying the computational optimality of this spectral algorithm.
    Namely, we show that, for a much broader range of signal distributions, whenever the spectral algorithm cannot detect a low-rank signal, then neither can any low-degree polynomial algorithm.
    This gives the first evidence for a computational hardness conjecture of Guionnet, Ko, Krzakala, and Zdeborov\'{a}~(2023).
    With similar techniques, we also prove sharp information-theoretic lower bounds for a class of signal distributions not treated by prior work.
    Unlike all of the above results on inhomogeneous models, our results do not assume that the variance profile has a block structure, and suggest that the same spectral algorithm might remain optimal for quite general profiles.
    We include a numerical study of this claim for an example of a smoothly-varying rather than piecewise-constant profile.
    Our proofs involve analyzing the \emph{graph sums} of a matrix, which also appear in free and traffic probability, but we require new bounds on these quantities that are tighter than existing ones for non-negative matrices, which may be of independent interest.
\end{abstract}

\thispagestyle{empty}
\clearpage

\thispagestyle{empty}
\tableofcontents

\clearpage

\pagenumbering{arabic}
\section{Introduction}

\subsection{Spiked Matrix Models}
\label{sec:homog-spiked}

Spiked matrix models are one of the fundamental examples of high-dimensional statistics, demonstrating in a relatively simple setting the important phenomena of sharp phase transitions in the effectiveness of natural estimators and statistical-to-computational gaps.
The simplest version of spiked matrix model is the \emph{spiked Wigner matrix model}, where, for $x \sim \sP_N$ for some probability measure $\sP_N$ over $\R^N$ normalized such that $\|x\| \approx \sqrt{N}$, we observe
\[ Y = \frac{\beta}{\sqrt{N}}xx^{\top} + W, \]
where $\beta > 0$ is a scalar parameter and $W_{ij} = W_{ji} \sim \sN(0, 1)$ independently for all $1 \leq i \leq j \leq N$.\footnote{Different prior works make different choices about how to treat the diagonal of $W$; this choice is almost always inconsequential in such models for the kinds of questions we will discuss.}
We call $x$ the \emph{signal} of such a model.
We observe this signal only through $Y$, and consider one of two computational tasks:
\begin{enumerate}
    \item \textbf{Hypothesis testing} between $Y \sim \P_N$ the law of $Y$ above with $\beta = c > 0$ for some given constant $c$ (which we call the \emph{planted} or \emph{alternative} model) and $Y \sim \Q_N$ the law of $Y$ with $\beta = 0$ (which we call the \emph{null} model).
    \item \textbf{Estimating} $x$ from the observation $Y$.
\end{enumerate}

We will mostly focus on hypothesis testing here, and introduce the following notions of success for that task.
The first is a sensible general definition.
\begin{definition}[Strong detection]
    We say that a sequence of functions $t_N: \mathbb{R}^{N \times N}_{\sym} \to \{0, 1\}$, which in this context we call \emph{tests}, achieves \emph{strong detection} in the above setting if 
    \[ \P_N[t_N(Y) = 1] + \mathbb{Q}_N[t_N(Y) = 0] = o(1). \]
\end{definition}
\noindent
The following more quantitative version is the one we will work with in our main results.
\begin{definition}[Strong separation]
    We say that a sequence of functions $f_N: \mathbb{R}^{N \times N}_{\sym} \to \R$ achieves \emph{strong separation} in the above setting if
    \[ \Ex_{Y \sim \P_N} f_N(Y) - \Ex_{Y \sim \Q_N} f_N(Y) = \omega\left(\sqrt{\Varx_{Y \sim \Q_N} f_N(Y)} + \sqrt{\Varx_{Y \sim \P_N} f_N(Y)} \right). \]
\end{definition}
\noindent
If $f_N(Y)$ achieve strong separation, then it is easy to check that setting $t_N(Y) \colonequals \bm 1\{f_N(Y) > \tau\}$ for a suitable threshold $\tau$ achieves strong detection, in essentially the same runtime as it takes to compute $f_N$.
Conversely, if $t_N(Y)$ achieve strong detection then they themselves also achieve strong separation.
Thus, the two notions are equivalent over arbitrary $f_N$ and $t_N$; however, our main results and some of the prior work we discuss below concern strong separation by special classes of $f_N$, which does not in general have a direct correspondence with strong detection by a special class of $t_N$.

Since for large $\beta$ one expects the rank-one deformation above to create an outlier eigenvalue in $Y$, it is natural to try to perform these tasks using the top eigenpair of $Y$.
We write $\lambda_1(Y) \geq \cdots \geq \lambda_N(Y)$ for the sorted eigenvalues of a matrix, and $v_i(Y)$ for the unit eigenvector associated to eigenvalue $\lambda_i(Y)$.
The following important result of random matrix theory, a variant of the celebrated \emph{Baik--Ben Arous--P\'{e}ch\'{e} (BBP) transition}, characterizes how well $\lambda_1(Y)$ and $v_1(Y)$ perform at testing and estimation, respectively.
We introduce a different normalization of $Y$,
\[ \what{Y} \colonequals \frac{1}{\sqrt{N}}Y, \]
which is useful as standard random matrix theory results imply that $\|\what{Y}\| = O(1)$ with high probability \cite{FK-1981-EigenvaluesRandomMatrices}.
Below, $\pto$ denotes convergence in probability of random variables to constants.
\begin{theorem}[\cite{FP-2007-LargestEigenvalueWigner,CDMF-2009-DeformedWigner}]
    \label{thm:bbp}
    In the above setting, suppose that $\|x\| / \sqrt{N} \pto 1$.
    Then, the following hold.
    \begin{itemize}
        \item If $0 \leq \beta \leq 1$, then
        \begin{align*}
        \lambda_1(\what{Y}) &\pto 2, \\
        \left|\left\langle v_1(\what{Y}), \frac{x}{\sqrt{N}}\right\rangle\right| &\pto 0.
        \end{align*}
        \item If $\beta > 1$, then
        \begin{align*}
            \lambda_1(\what{Y}) &\pto \beta + \frac{1}{\beta} > 2, \\
            \left|\left\langle v_1(\what{Y}), \frac{x}{\sqrt{N}}\right\rangle\right| &\pto \sqrt{1 - \frac{1}{\beta^2}} > 0.
        \end{align*}
    \end{itemize}
\end{theorem}
\noindent
Thus the theorem implies that $\lambda_1(\what{Y})$ can be used to hypothesis test with high probability of success (indeed, it can be shown with a bit more argument that $\lambda_1(\what{Y})$ achieves strong separation in this setting) and $v_1(\what{Y})$ can be used to estimate $x$ non-trivially if and only if 
\[ \beta > \beta_* \colonequals 1. \]

It is natural to ask whether this performance is optimal or not: maybe a better statistical procedure can achieve the above results for smaller $\beta$.
The answer depends on $\sP_N$, the ``prior'' distribution on $x$.
This distribution (though not $x$ itself) is assumed to be known to the statistical algorithm under consideration.
So, for instance, if $\sP_N = \delta_{x_{\star}}$ is concentrated on a single value $x_{\star}$, then of course estimation is trivially simple since an algorithm can just output $x_{\star}$, and one can show that hypothesis testing is also easy by thresholding $x_{\star}^{\top}Yx_{\star}$.
However, the following result shows that, for ``sufficiently random'' $\sP_N$, in fact no procedure can improve on the simple spectral algorithm above.
\begin{theorem}[\cite{PWBM-2018-PCAI}]
    \label{thm:info-hom}
    Suppose that either $\sP_N = \Unif(\{x \in \R^N: \|x\| = \sqrt{N}\})$ or $\sP_N = \Unif(\{\pm 1\}^N)$.
    If $\beta < \beta_* = 1$, then there is no sequence of tests $t_N$ achieving strong detection and no sequence of functions $f_N$ achieving strong separation in the above model.
\end{theorem}

This, too, can break down for certain $\sP_N$.
For instance, a long line of work has demonstrated that, if $\sP_N$ is supported on vectors with $p N$ non-zero entries for sufficiently small $p$ (the setting of \emph{sparse PCA}), then brute force searches requiring time $\exp(\Omega(N))$ over sparse vectors can succeed at hypothesis testing and estimation for certain values of $\beta < 1$, while conjecturing that this cannot be done by efficient algorithms \cite{LKZ-2015-LowRankChannelUniversality,LKZ-2015-PhaseTransitionsSparsePCA,KXZ-2016-MutualInformationMatrix,BDMKLZ-2016-MutualInformationReplica,BMVVX-2018-InfoTheoretic,AKJ-2018-SpikedWigner,LM-2019-LimitsLowRankEstimation}.
Corroborating these conjectures, the following result says that, for an even wider range of priors $\sP_N$, the threshold $\beta = 1$ is the best for a large class of algorithms, at least for hypothesis testing (see also \cite{MW-2023-PreciseErrorRatesPCA} for related but more intricate results on estimation).
Specifically, these results and the new ones we prove below consider the class of \emph{low-degree polynomial} separating statistics.
\begin{theorem}[\cite{kunisky2019notes}]
    \label{thm:low-deg-hom}
    Suppose that $\sP_N = \pi^{\otimes N}$ for some $\pi$ a bounded probability measure on $\R$ with mean 0 and variance 1.
    If $\beta < \beta_* = 1$, then there is no sequence of polynomials $f_N$ with $\deg(f_N) = o(N / \log N)$ that achieve strong separation in the above setting.
\end{theorem}
\noindent
Since a polynomial of degree $D$ in the roughly $N^2$ variables of $Y$ requires time $N^{\Theta(D)}$ to evaluate naively, this suggests that, taking polynomials as a proxy for all functions computable under a given runtime budget, even when it is possible as in sparse PCA, it should take nearly exponential time in $N$ to compute a strongly separating statistic when $\beta < \beta_*$.
We will subscribe in this paper to the belief that this kind of heuristic reasoning about low-degree polynomial separating statistics is accurate; as detailed in the surveys \cite{kunisky2019notes,Wein-2025-LowDegreeSurvey}, it has given correct predictions of computational cost and hardness for numerous similar problems in prior work.

\subsection{Inhomogeneous Spiked Matrix Models}

The spiked Wigner matrix model is unsatisfying in various ways.
One might object that one assumes that the noise is Gaussian, i.i.d., and applied additively to the signal, all of which might be inaccurate assumptions for modeling various statistical scenarios.
Recent work has begun to generalize this model to allow for various other noise distributions; in addition to the references we discuss in detail below, see \cite{BGN-2011-PerturbationsRandomMatrices,CDMFF-2011-DeformedWignerFreeConvolution,CDM-2016-DeformedRandomMatricesFreeProbability,barbier2022price,BCMS-2023-FundamentalLimitsStructuredPCA,zhang2024matrix,barbier2025information} for some other generalizations.

One of these directions has been to consider a noise matrix $W$ whose entries are still independent and Gaussian, but where the variances are \emph{inhomogeneous}, being allowed to vary from entry to entry.
We summarize this in the following model, which will be the subject of this paper.
We also relax the precise structural assumption on the rank-one signal from the above discussion, instead allowing for now the signal to be an arbitrary matrix (later we will consider it being both rank-one and more generally low-rank).
\begin{definition}[Inhomogeneous spiked Wigner matrix model]
    For each $N \geq 1$, let $\sX_N$ be a probability measure on $\R^{N \times N}_{\sym}$ and let $\Delta = \Delta^{(N)} \in \left(\R_{> 0} \cup \{+\infty\}\right)^{N \times N}_{\sym}$.
    Define
    \[ \fJ = \fJ^{(N)} \colonequals \{(i, j): 1 \leq i \leq j \leq N, \Delta_{ij} \neq +\infty\}. \]
    Associated to these parameters, define two probability measures $\Q_N$ and $\P_N$ on $\mathbb{R}^{\fJ}$:
    \begin{enumerate}
    \item To draw $Y \sim \mathbb{Q}_N$, draw $Y_{ij} = Y_{ji} \sim \mathcal{N}(0, \Delta_{ij})$ independently for each $(i, j) \in \fJ$.
    \item To draw $Y \sim \mathbb{P}_N$, first draw $X \sim \mathcal{X}_N$, and then draw $Y_{ij} = Y_{ji} \sim \sN(X_{ij}, \Delta_{ij})$ independently for each $(i, j) \in \fJ$.
    \end{enumerate}
    We call $\Delta$ the \emph{variance profile} of the model.
    We call the model \emph{finite-variance} if $\Delta_{ij} \neq +\infty$ for all $i, j \in [N]$, in which case it may be viewed as a pair of probability measures on $\R^{N \times N}_{\sym}$
    \label{def:inhom-wigner}
\end{definition}
\noindent
We allow for some entries to be observed with ``infinite variance'' of noise, by which in practice we just mean that those entries are not observed at all and we only observe a subset $\fJ$ of entries of a symmetric matrix.
Some other work has interpreted such models as having ``censorship'' of some entries (e.g., \cite{ABBS-2014-ExactRecoveryCensoredSBM,SLKZ-2015-SpectralDetectionCensoredSBM}), but we will see that, formally, it is consistent and convenient for us to view these entries as indeed having noise of variance equal to $+\infty$.

Note that if we take $\Delta_{ij} = 1$ for all $i, j \in [N]$ and $X \sim \sX_N$ drawn as $X = \beta xx^{\top}$ for $x \sim \sP_N$, then we recover the homogeneous model discussed above in Section~\ref{sec:homog-spiked}.
$\Q_N$ will serve as a null model in this general setting; in the rank-one case above, it is what we would obtain if we took $\beta = 0$.
Such models seem to have first appeared in the concurrent series of works \cite{BR-2020-InfoTheoryMultiviewSpikedMatrix,BR-2022-InfoTheoryInhomogeneousSpikedMatrix} and \cite{ACCM-2021-DeepBoltzmannNishimori,ACCM-2021-MultiSpeciesMeanFieldSpinGlass,ACCM-2022-PhysicsMultiChannelSpikedWigner}, the former motivated by statistics and the latter by statistical physics.
Further important results on this model that we will discuss below have been obtained in the recent works \cite{pak2024optimal,MKK-2024-OptimalPCABlockSpikedMatrix,BCSVH-2024-MatrixConcentrationFreeProbability2,GKKZ-2025-InhomogeneousSpikedWigner}.

Previous work on such statistical models appears to have nearly exclusively focused on the case of finite variances having the following special \emph{block structure} of the variance profile (the only exception we are aware of is \cite{bigot2021freeness}, which as we will see studies a spectral algorithm that is likely suboptimal).
In contrast, all of our results allow for infinite variances, and some of our results also do not require a block-structured variance profile.
General variance profiles are actually quite frequently considered in random matrix theory outside of statistical applications; see Section~\ref{sec:related} for references to such results.

\begin{definition}[Block-structured matrices]
\label{BCM}
    Suppose that $A^{(N)} \in \left(\R_{> 0} \cup \{+\infty\}\right)^{N \times N}_{\sym}$ for each $N 
    \geq 1$.
    Suppose that there exist $n \geq 1$, $\rho_1, \dots, \rho_n \in (0, 1)$ with $\rho_1 + \cdots + \rho_n = 1$, and a matrix $\bar{A} \in \left(\R_{> 0} \cup \{+\infty\}\right)^{n \times n}_{\sym}$, such that the following hold: for each $N$, there is a function $q: [N] \to [n]$ so that $A^{(N)}_{ij} = \bar{A}_{q(i)q(j)}$ for all $i, j \in [N]$, and for each $i \in [n]$, $|q^{-1}(i)| / N \to \rho_i$ as $N \to \infty$.
    Then, we say that the sequence $A^{(N)}$ is \emph{block-structured} with parameters $(\bar{A}, \rho_1, \dots, \rho_n)$.
    
    When $\Delta^{(N)}$ is block-structured for an inhomogeneous spiked Wigner matrix model, we also call the model itself block-structured.
    We always use a bar to denote the matrix parameter associated to a block-structured sequence.
    For instance, we will talk about matrices $\Delta^{(N)}$ and $\Phi^{(N)}$ (the latter to be defined below) being block-structured, and $\bar{\Delta}$ and $\bar{\Phi}$ will be the associated parameters.
    \label{def:block-structured}
\end{definition}
\noindent
It seems that the reason for the focus on block-structured models in prior work is that, as the definition makes clear, they have a well-defined entrywise limiting behavior.
Our results will show, however, that such a strong notion of limit for the variance profile sequence need not hold to draw useful conclusions about the tractability of the associated spiked models.

In either the general inhomogeneous spiked Wigner matrix model or its block-structured special case, it is natural to ask all of the questions we have mentioned above in Section~\ref{sec:homog-spiked}: what strength of signal is required for testing or estimation to be tractable, either by arbitrary functions or computationally efficient ones?
Perhaps one expects that, if the signal $X \sim \sX_N$ is low-rank and positive semidefinite (like $X = \frac{\beta}{\sqrt{N}} xx^{\top}$ from our earlier discussion), then it should be a good idea to simply use $\lambda_1(\what{Y})$ and $v_1(\what{Y})$ again.
Or, perhaps one should ``whiten'' the noise and divide each entry $Y_{ij}$ by $\sqrt{\Delta_{ij}}$ before computing the spectral decomposition.
Although either approach succeeds for sufficiently large $\beta$, it turns out that neither is optimal in the inhomogeneous case (see discussion in \cite{GKKZ-2025-InhomogeneousSpikedWigner,MKK-2024-OptimalPCABlockSpikedMatrix} as well as our Section~\ref{sec:numerical}).

Motivated by the analysis of approximate message passing, \cite{pak2024optimal} proposed another spectral algorithm, as follows.
We introduce the entrywise reciprocal of $\Delta$, which will be useful throughout our discussion:
\[ \Phi = \Phi^{(N)} \colonequals \Delta^{\odot -1}, \,\,\,\, \text{i.e., } \,\,\,\, \Phi_{ij} = \left\{\begin{array}{ll} 1 / \Delta_{ij} & \text{if } \Delta_{ij} \neq +\infty, \\ 0 & \text{if } \Delta_{ij} = +\infty\end{array}\right\}. \]
Here and throughout, $\odot$ denotes the Hadamard entrywise product or power of matrices.
Note that we work here and always under the convention that
\[ \frac{1}{+\infty} = 0, \]
and with this convention $\Phi^{(N)} \in (\R_{\geq 0})^{N \times N}_{\sym}$, never having infinite entries even if $\Delta^{(N)}$ does.

In our normalization, where $\|x\| / \sqrt{N} \pto 1$ for $x \sim \sP_N$ and $X \sim \sX_N$ is $X = \frac{\beta}{\sqrt{N}} xx^{\top}$ for such $x$ and some $\beta > 0$, the algorithm proposed by \cite{pak2024optimal} uses the matrix
\begin{equation}
\label{eq:tildeY}
\sH(Y) \colonequals \frac{\beta}{\sqrt{N}} \Phi \odot Y - \beta^2\, \diag\left(\frac{1}{N}\Phi 1_N\right), 
\end{equation}
where $1_N$ denotes the all-ones vector.
The algorithm is quite unusual on its face, as the first term, in our above terminology, seems to ``over-whiten'' the observation $Y$.
We remark that, unlike both alternatives discussed above, to implement this algorithm requires knowledge of the parameter $\beta$ (this point is perhaps less clear in the notation used by the previous works \cite{pak2024optimal,MKK-2024-OptimalPCABlockSpikedMatrix,BCSVH-2024-MatrixConcentrationFreeProbability2} studying this algorithm, who do not introduce $\beta$ as a separate parameter, though \cite{pak2024optimal} make the same point in their Remark 3.1).
The analysis of a spectral algorithm using this matrix, as conjectured by \cite{pak2024optimal} and proved concurrently in the two references below, is as follows.
\begin{theorem}[\cite{MKK-2024-OptimalPCABlockSpikedMatrix,BCSVH-2024-MatrixConcentrationFreeProbability2}]
    \label{thm:bbp-inhom}
    Consider a block-structured finite-variance inhomogeneous spiked Wigner matrix model and suppose that the block structure of $\Delta$ has parameters $(\bar{\Delta}, \rho_1, \dots, \rho_n)$.
    Suppose that $x \sim \sP_N$ has i.i.d.\ entries of mean zero and variance 1, and that $X \sim \sX_N$ is $X = \frac{\beta}{\sqrt{N}} xx^{\top}$ for $x$ as above and some $\beta > 0$.
    Write $\bar{\Phi} \colonequals \bar{\Delta}^{\odot -1}$, and define
    \begin{equation}
    \beta_* = \beta_*(\bar{\Delta}, \rho_1, \dots, \rho_n) \colonequals \sqrt{\frac{1}{\lambda_1(\Diag(\rho)^{1/2}\, \bar{\Phi}\, \Diag(\rho)^{1/2})}}. \label{eq:beta-crit-block}
    \end{equation}
    Then, the following hold:
    \begin{itemize}
        \item If $\beta \leq \beta_*$, then
        \begin{align*}
        \lambda_1(\sH(Y)) - \lambda_2(\sH(Y)) &\pto 0, \\
        \left|\left\langle v_1(\sH(Y)), \frac{x}{\sqrt{N}}\right\rangle\right| &\pto 0.
        \end{align*}
        \item If $\beta > \beta_*$, then there exist $\delta = \delta(\beta) > 0$ and $\epsilon = \epsilon(\beta) > 0$ such that
        \begin{align*}
        \lambda_1(\sH(Y)) - \lambda_2(\sH(Y)) &\pto \delta > 0, \\
        \left|\left\langle v_1(\sH(Y)), \frac{x}{\sqrt{N}}\right\rangle\right| &\pto \epsilon > 0.
        \end{align*}
    \end{itemize}
\end{theorem}
\noindent
We note that, while the result above is restricted to finite variances, the quantity $\beta_*$ defined in~\eqref{eq:beta-crit-block} is still well-defined even if some variances are infinite, which we will use below.
The nature of and reason for this threshold $\beta_*$ should still be mysterious; we will clarify how this value arises when we discuss our main results and analysis below.

For now, two basic observations will suffice.
First, note that the values in $\bar{\Phi}$ are the reciprocals of the noise variances, so they get smaller as the amount of noise increases, and thus $\beta_*$ gets larger as the amount of noise increases, meaning that the problem becomes harder (at least for this spectral algorithm), as we expect.
Second, note that the result is very similar to Theorem~\ref{thm:bbp}.
The reason for the slightly different form of the result on $\lambda_1(\sH(Y))$ compared to Theorem~\ref{thm:bbp} on $\lambda_1(\what{Y})$ is that, unlike for $\what{Y}$, the shape of the ``bulk'' eigenvalue distribution of $\sH(Y)$ depends on $\beta$ by construction.
So, our condition on the top eigenvalue captures not whether $\lambda_1(\sH(Y))$ is unusually large compared to the case $\beta = 0$, but whether it is an outlier from this bulk distribution.
This can still be used to achieve strong detection when $\beta > \beta_*$ by setting an appropriate threshold for $\lambda_1(\sH(Y))$ depending on $\beta$.
Actually, Theorem~\ref{thm:bbp-inhom} generalizes Theorem~\ref{thm:bbp}, since for $\Delta$ the all-ones matrix we recover the threshold $\beta_* = 1$, and while we have not stated them, the full results of \cite{MKK-2024-OptimalPCABlockSpikedMatrix} give explicit though complicated descriptions of $\delta$ and $\epsilon$ appearing above, as well as the high-probability limits of $\lambda_i(\sH(Y))$ for $i \in \{1, 2\}$, making the result equally precise to Theorem~\ref{thm:bbp}.
Thus, just as in that case, Theorem~\ref{thm:bbp-inhom} should be read as saying that spectral algorithms using $\sH(Y)$ succeed if and only if $\beta > \beta_*$.

\subsection{Main Results}
\label{sec:results}

Our goal will be to probe whether the above algorithm---which, without some study of its derivation via approximate message passing, appears quite unnatural---is optimal.
The threshold $\beta_*$ in the block-structured case is also a natural one for the performance of approximate message passing algorithms, and on this basis \cite{GKKZ-2025-InhomogeneousSpikedWigner} conjectured that efficient algorithms should not be able to hypothesis test or estimate effectively when $\beta < \beta_*$.

As in Theorem~\ref{thm:info-hom} on homogeneous spiked Wigner models cited above, only in some cases will we be able to show that this algorithm is \emph{statistically} or \emph{information-theoretically} optimal, meaning that no procedure regardless of runtime can achieve strong detection when $\beta < \beta_*$.
These results of ours will generalize Theorem~\ref{thm:info-hom} to the inhomogeneous setting, complementing and extending some information-theoretic lower bounds of \cite{GKKZ-2025-InhomogeneousSpikedWigner}, which apply only to the block-structured case and then only in one very special case mentioned below match the threshold $\beta_*$.

Paralleling Theorem~\ref{thm:low-deg-hom} for the homogeneous case, we will also give evidence by analyzing low-degree polynomials that, for a much wider range of prior distributions $\sP_N$, this algorithm is \emph{computationally} optimal, suggesting that no polynomial-time procedure can hypothesis test when $\beta < \beta_*$ and giving evidence for the conjecture of \cite{GKKZ-2025-InhomogeneousSpikedWigner}.
Notably, going beyond that conjecture, our results will treat very general variance profiles which need not be block-structured, and will suggest that in fact a threshold generalizing the above $\beta_*$ might be the correct computational threshold for a broader class of inhomogeneous spiked Wigner matrix models.

A particularly important role will turn out to be played by the largest eigenvalue of $\Phi$.
Since the entries of $\Phi$ are non-negative, this eigenvalue is strictly positive unless we are in the trivial situation where $\Phi = 0$, which only arises when $\fJ = \emptyset$ and we make no observation at all.
Recall that we work in a scaling where the entries of $\Phi$ have magnitude $\Phi_{ij} = \Theta(1)$, whereby we expect the largest eigenvalue to be of size $\Theta(N)$.
We have the following description of the threshold $\beta_*$ in terms of this eigenvalue, which we prove in Section~\ref{sec:linalg}.
\begin{proposition}
    \label{prop:mu1-block}
    In a block-structured inhomogeneous spiked Wigner model,
    \[ \widehat{\mu}_1 \colonequals \lim_{N \to \infty} \frac{\lambda_1(\Phi^{(N)})}{N} = \lambda_1(\Diag(\rho)^{1/2}\, \bar{\Phi}\, \Diag(\rho)^{1/2}). \]
    This is related to $\beta_*$ defined for such models in \eqref{eq:beta-crit-block} by
    \[ \beta_* = \frac{1}{\sqrt{\what{\mu}_1}}. \]
\end{proposition}

We now give our main result on block-structured models.
In this case, we give a computational lower bound under mild conditions on the prior distribution $\sX$, a statistical lower bound under stronger conditions, and allow for signals of any constant rank $\kappa$.
Here and below, we apply the definitions of strong detection and strong separation to the probability measures $\P_N$ and $\Q_N$ described in  Definition~\ref{def:inhom-wigner}.
\begin{theorem}[Block-structured variance profile]
\label{theorem:comp-lower-bound-block-constant-variance}
Consider the inhomogeneous spiked Wigner matrix model (Definition~\ref{def:inhom-wigner}) with block-structured (but possibly infinite) variances.
Suppose that $\sP_N$ is a probability measure over $\R^{N \times \kappa}$ such that $V \sim \sP_N$ has i.i.d.\ rows drawn from some probability measure $\nu$ on $\R^{\kappa}$.
Let $X \sim \sX_N$ be $X = \frac{\beta}{\sqrt{N}} VV^{\top}$ for $V \sim \sP_N$ and some $\beta > 0$.
Then, the following hold, for $\beta_*$ as defined in \eqref{eq:beta-crit-block}:
\begin{enumerate}
    \item \textbf{Computational lower bound:} Suppose that $\nu$ satisfies the conditions
    \begin{enumerate}
        \item $\E_{x \sim \nu}[x] = 0$.
        \item $\|\Cov_{x \sim \nu}[x]\| = \|\Ex_{x \sim \nu} xx^{\top}\| \leq 1$.
        \item $\nu$ has bounded support.
    \end{enumerate}
    If $\beta < \beta_*$, then there is no sequence of polynomials $f_N \in \R[Y]$ with $\deg(f_N) = o(N / \log N)$ that achieve strong separation.
    \item \textbf{Statistical lower bound:} Suppose that the above conditions on $\nu$ are satisfied, as well as that
    \begin{enumerate}[resume]
        \item When $x, x^{\prime} \sim \nu$ are i.i.d., then $x \otimes x^{\prime}$ is 1-moment-subgaussian (see Definition~\ref{def:moment-subgaussian}).
    \end{enumerate}
    If $\beta < \beta_*$, then there is no sequence of functions $f_N$ that achieve strong separation, there is no sequence of tests $t_N$ that achieve strong detection, and $\chi^2(\P_N \mid \Q_N) = O(1)$.\footnote{As we have mentioned, the first two statements here are equivalent, and as we state in Proposition~\ref{prop:chi-squared-testing}, both are implied by the third.}
\end{enumerate}
\end{theorem}
\noindent
We make a few remarks on these results.
Part 1 of the result generalizes Theorem~\ref{thm:low-deg-hom} to the inhomogeneous setting, and Part 2 does the same for Theorem~\ref{thm:info-hom}.
As mentioned above, the conclusion of Part 1 should be read as suggesting that nearly exponential time is required to compute a strongly separating statistic.
Our setting is a little bit more general than both those results as stated above and the setting treated by \cite{pak2024optimal}, because we allow for higher-rank signals.
This setting was also used (to prove different kinds of results) by \cite{GKKZ-2025-InhomogeneousSpikedWigner}, and for example Theorem~\ref{thm:low-deg-hom} for computational lower bounds in the homogeneous case was extended to higher rank in \cite{bandeira2021spectral}.
Lastly, a simple example of $\nu$ satisfying Condition~(d) of Part~2 is, in the case $\kappa = 1$, $\nu = \Unif(\{\pm 1\})$ (as included in Theorem~\ref{thm:info-hom} for the homogeneous case).
We note that \cite{GKKZ-2025-InhomogeneousSpikedWigner} (specifically, their Lemma~2.15) also prove statistical lower bounds along the lines of Part 2 of our result, except pertaining to estimation and the minimum mean squared error achievable by any estimator.
In general the threshold they obtain for $\beta$ for these lower bounds to hold is smaller than $\beta_*$, except in the very special case $\kappa = 1$ and $\nu = \sN(0, 1)$ (which, however, our result does not cover).

We next give our results for the case of general variance profiles.
In this case, our first description of $\beta_*$ is no longer sensible, since we do not have access to the block structure parameters $\bar{\Phi}$ and $\rho_i$ anymore.
But, we may mimic the formulation of Proposition~\ref{prop:mu1-block}.
To that end, we introduce a mild structural assumption, which we always make from now on without further mention:
\begin{assumption}
    \label{assumptionA}
    In an inhomogeneous spiked Wigner matrix model with general variance profile (possibly having infinite entries), the following limit exists:
    \[ \what{\mu}_1 \colonequals \lim_{N \to \infty} \frac{\lambda_1(\Phi^{(N)})}{N} > 0. \]
\end{assumption}
\noindent
If this is not the case, then the $\liminf$ and $\limsup$ of the above sequence are different, which our results will show suggests that the easiest subsequence of spiked matrix models in a given sequence is considerably easier than the hardest subsequence.
Thus, in this situation, we suggest that the hardness of the sequence taken together is unclear and one should consider one or both of these subsequences individually instead.
Finally, under Assumption~\ref{assumptionA}, we define a generalized critical $\beta$,
\begin{equation} 
\beta_* \colonequals \frac{1}{\sqrt{\what{\mu}_1}} = \lim_{N \to \infty} \sqrt{\frac{N}{\lambda_1(\Phi^{(N)})}}. \label{eq:beta-crit-general}
\end{equation}
By Proposition~\ref{prop:mu1-block}, this indeed generalizes the threshold from the block-structured case.

To handle the setting of general variance profiles, we slightly restrict the class of signals we consider: instead of low-rank signal matrices formed as the Gram matrices of i.i.d.\ vectors, we consider rank-one signals formed as $xx^{\top}$ for $x$ having i.i.d.\ entries.
Also, we introduce the mild but technically important assumption that the distribution of those entries of $x$ is \emph{symmetric}.
It seems likely that both of these assumptions can be relaxed, but also we expect this to require a considerably different technical approach.
We include some discussion of this following our proof in Section~\ref{sec:pf:theorem:general-bound}.

\begin{theorem}[General variance profile]
\label{theorem:general-bound}
Consider the inhomogeneous spiked Wigner matrix model with arbitrary variance profile (possibly both non-block-structured and having some infinite entries).
Suppose that $\sP_N$ is a probability measure over $\R^{N}$ such that $x \sim \sP_N$ has i.i.d.\ entries drawn from some probability measure $\nu$ on $\R$.
Let $X \sim \sX_N$ be $X = \frac{\beta}{\sqrt{N}} xx^{\top}$ for $x \sim \sP_N$ and some $\beta > 0$.
Then, the following hold:
\begin{enumerate}
    \item \textbf{Computational lower bound:} Suppose that $\nu$ satisfies the conditions:
    \begin{enumerate}
        \item $\E_{x \sim \nu}[x] = 0$.
        \item $\E_{x \sim \nu}[x^2] \leq 1$.
        \item $\nu$ is symmetric (i.e., $x \sim \nu$ has the same law as $-x$).
        \item If $x, x^{\prime} \sim \nu$ are independent, then $xx^{\prime}$ is $\sigma^2$-subgaussian for some $\sigma^2 > 0$.
    \end{enumerate}
    Suppose also that the variance profiles $\Delta^{(N)}$ satisfy the conditions, for some $D(N) \in \mathbb{N}$,
    \begin{enumerate}[resume]
        \item There exists a constant $B > 0$ such that, for all $N \geq 1$ and $1 \leq i, j \leq N$, $\Delta^{(N)}_{ij} \geq \frac{1}{B}$.
        \item $D(N) N^2 / \lambda_1(\Phi^{(N)})^3 \to 0$ as $N \to \infty$.
    \end{enumerate}
    If $\beta < \beta_*$, then there is no sequence of polynomials $f_N \in \R[Y]$ with $\deg(f_N) \leq D(N)$ that achieve strong separation.
    \item \textbf{Statistical lower bound:} Suppose that Conditions~(a), (b), (c), and (e) above are satisfied, as well as:
    \begin{enumerate}
        \item[(d\,$^{\prime}$)]
        When $x, x^{\prime} \sim \nu$ are i.i.d., then $xx^{\prime}$ is 1-moment-subgaussian (see Definition~\ref{def:moment-subgaussian}).
    \end{enumerate}
    If $\beta < \beta_*$, then there exists no sequence of functions $f_N$ that achieve strong separation, there exists no sequence of functions $t_N$ that achieve strong detection, and $\chi^2(\P_N \mid \Q_N) = O(1)$.
\end{enumerate}
\end{theorem}

We have stated the assumptions to make the results more general, but let us point out two simplifications in typical situations.
First, if the measure $\nu$ is bounded (as in Theorem~\ref{theorem:comp-lower-bound-block-constant-variance}), then the subgaussianity in Condition~(d) is automatically satisfied.
Condition (d$^{\prime}$), on the other hand, may be viewed as a variation on a measure being 1-subgaussian, defined in terms of bounds on each moment rather than the moment generating function (see Proposition~\ref{prop:subgauss-moments} and surrounding discussion).
Second, if we have a uniform upper bound on the variance profile, say that for some $B > 0$ not depending on $N$ we have
\begin{equation} 
\label{eq:uniform-profile-bounds}
\frac{1}{B} \leq \Delta^{(N)}_{ij} \leq B \text{ for all } N \geq 1 \text{ and } i, j \in [N],
\end{equation}
then the same bound holds for the entries of $\Phi^{(N)}$, whereby $\lambda_1(\Phi^{(N)}) \geq \frac{1}{B}N$, and we may take any degree sequence $D(N) = o(N)$ in Condition (f).
But, we emphasize that our result is considerably more general and still gives non-trivial lower bounds for much more uneven variance profiles.
For instance, even if 
 we only have $\Delta^{(N)}_{ij} \leq N^{1/3 - \delta}$ for all $N \geq 1$ and $i, j \in [N]$ for some $\delta > 0$, then by the same token we have $\lambda_1(\Phi^{(N)}) \geq N^{2/3 + \delta}$ and we obtain a lower bound against strong separation by degree $D(N) = o(N^{3\delta})$ polynomials.
Following our previous heuristic discussion, this suggests that time close to $\exp(\Omega(N^{3\delta}))$ is required to compute a strongly separating statistic, and in particular that this cannot be done in polynomial time.

\subsection{Additional Results}

\subsubsection{Growth of Low-Degree \texorpdfstring{$\chi^2$-}{Chi-Squared }Divergence}
\label{sec:low-deg-div}

To prove our computational lower bounds, we will use that degree $D = D(N)$ polynomials cannot achieve strong separation provided that a quantity we call the \emph{low-degree $\chi^2$-divergence} and denote $\chi^2_{\leq D}(\P_N \mid \Q_N) = O(1)$ as $N \to \infty$ (see Definition~\ref{def:low-deg-div} and Proposition~\ref{prop:chi-squared-testing}).
It is tempting to view the opposite case, where $\chi^2_{\leq D}(\P_N, \Q_N) \to \infty$, as evidence that low-degree polynomials \emph{do} achieve strong separation.
Unfortunately, this is not the case: the low-degree $\chi^2$-divergence is only a ``one-sided'' quantity, and its divergence only ensures that $$\Ex_{Y \sim \P_N} f_N(Y) - \Ex_{Y \sim \Q_N} f_N(Y) = \omega\left(\sqrt{\Varx_{Y \sim \Q_N} f_N(Y)} \right).$$
In particular, it is possible that $$\Ex_{Y \sim \P_N} f_N(Y) - \Ex_{Y \sim \Q_N} f_N(Y) = O\left(\sqrt{\Varx_{Y \sim \P_N} f_N(Y)} \right),$$
in which case strong separation does not hold because $f_N$ fluctuates too much under the planted model.
It is not hard to construct pathological situations where this happens, but in fact this behavior has also been observed in some reasonable models in recent work \cite{BAHSWZ-2022-FranzParisiLowDegree,COGHWZ-2022-PhaseTransitionsGroupTesting,DMW-2023-DenseSubhypergraphsLowDegree,DDL-2023-GraphMatchingLowDegree}.

Still, the low-degree $\chi^2$-divergence being unbounded gives some partial evidence that low-degree polynomials or otherwise tractable functions may achieve strong separation and detection, and at the very least shows that our analysis of the divergence is tight, so we give converse results in variants of the settings of Theorems~\ref{theorem:comp-lower-bound-block-constant-variance} and~\ref{theorem:general-bound} stating that this happens when $\beta > \beta_*$.
Of course, it would be more convincing to show that a specific low-degree polynomial actually achieves strong separation, but we expect such polynomials to be related to the spectral algorithm using $\sH(Y)$, whose behavior at least in the case of general variance profiles (the setting of Theorem~\ref{thm:general-low-deg-div} below) is not yet well-understood, as we briefly discuss in Section~\ref{sec:numerical}.
\begin{theorem}
\label{thm:block-low-deg-div}
Consider a block-structured inhomogeneous spiked Wigner matrix model.
Suppose that $\sP_N$ is a probability measure over $\R^{N \times \kappa}$ such that $V \sim \sP_N$ has i.i.d.\ rows drawn from some probability measure $\nu$ on $\R^{\kappa}$.
Let $X \sim \sX_N$ be $X = \frac{\beta}{\sqrt{N}} VV^{\top}$ for $V \sim \sP_N$ and some $\beta > 0$.
Suppose that $\nu$ satisfies the conditions:
\begin{enumerate}
        \item[(a)] $\E_{x \sim \nu}[x] = 0$.
        \item[(b)] $\|\Cov_{x \sim \nu}[x]\| = \|\Ex_{x \sim \nu} xx^{\top}\| = 1$.
        \item[(c)] $\nu$ has bounded support.
    \end{enumerate}
    If $\beta > \beta_*$ and $D = D(N) \to \infty$ as $N \to \infty$, then $\chi^2_{\leq D}(\P_N \mid \Q_N) \to \infty$ as $N \to \infty$.
\end{theorem}

\begin{theorem}
\label{thm:general-low-deg-div}
Consider an inhomogeneous spiked Wigner matrix model with arbitrary variance profile.
Suppose that $\sP_N$ is a probability measure over $\R^{N}$ such that $x \sim \sP_N$ has i.i.d.\ entries drawn from some probability measure $\nu$ on $\R$.
Let $X \sim \sX_N$ be $X = \frac{\beta}{\sqrt{N}} xx^{\top}$ for $x \sim \sP_N$ and some $\beta > 0$.
Suppose that $\nu$ satisfies the conditions:
    \begin{enumerate}
        \item[(a)] $\E_{x \sim \nu}[x] = 0$.
        \item[(b)] $\E_{x \sim \nu}[x^2] = 1$.
        \item[(c)] $\nu$ has bounded support.
    \end{enumerate}
    Suppose also that the reciprocal variance profiles $\Phi^{(N)}$ satisfy the conditions:
    \begin{enumerate}[resume]
        \item[(d)] $\lim_{N \to \infty} \max_{i \in [N]} \|v_i(\Phi^{(N)})\|_{\infty} = 0$, where $v_i(\cdot)$ are the unit eigenvectors of a symmetric matrix.
        \item[(e)] Either $\Phi^{(N)} \succeq 0$ for all $N \geq 1$, or $\rank(\Phi^{(N)}) \leq K$ for all $N \geq 1$ for some $K$ not depending on $N$ and the limits $\what{\mu}_i \colonequals \lim_{N \to \infty} \lambda_i(\Phi^{(N)}) / N$ exist for each $i = 1, \dots, K$.
    \end{enumerate}
    If $\beta > \beta_*$ and $D = D(N) \to \infty$ as $N \to \infty$, then $\chi^2_{\leq D}(\P_N \mid \Q_N) \to \infty$ as $N \to \infty$.
\end{theorem}

\subsubsection{Statistical-to-Computational Gaps}

Our results imply that the important phenomenon of \emph{statistical-to-computational gaps} in some models of low-rank estimation persists in the case of inhomogeneous noise.
We do not go into the details, but sketch the idea of this implication here.
Consider as an illustrative example the \emph{sparse Rademacher} prior,
\[ \nu = \frac{p}{2} \delta_{-1 / \sqrt{p}} + (1 - p) \delta_0 + \frac{p}{2}\delta_{1 / \sqrt{p}} \]
for some constant $p \in (0, 1)$ not depending on $N$.
For any constant $p$ this $\nu$ is bounded and has mean 0 and variance 1, and thus all of our computational lower bounds apply to signals $x \sim \sP_N = \nu^{\otimes N}$.
For a simple concrete setting, consider this choice with a variance profile that is uniformly bounded above and below, as in \eqref{eq:uniform-profile-bounds}.
Then (and under Assumption~\ref{assumptionA}), Theorem~\ref{theorem:general-bound} implies that polynomials of degree $o(N)$ cannot detect such a sparse signal whenever $\beta < \beta_*$, where we emphasize that $\beta_*$ depends only on the sequence of variance profiles, not on $p$.
More heuristically, this suggests that to detect such a signal requires nearly exponential time in $N$.

In the homogeneous case, \cite{BMVVX-2018-InfoTheoretic} showed that, in fact, there is a $p_0 \in (0, 1)$ such that, whenever $p < p_0$, then a brute force exponential-time computation of the likelihood ratio \emph{does} successfully distinguish $\Q_N$ from $\P_N$ for $\beta \in (\beta_*^{\prime}, \beta_*)$ for some $\beta_*^{\prime} < \beta_*$.
Thus, in these cases, there is a \emph{statistical-to-computational gap}: for certain parameter regimes, it is possible to distinguish the null and planted distributions, but, based on our analysis of low-degree polynomials, we believe it is only possible to do so very inefficiently (see the discussion of this line of work on sparse PCA above in Section~\ref{sec:homog-spiked} for more references).

It is tedious but straightforward to reproduce the analysis of \cite{BMVVX-2018-InfoTheoretic} for the likelihood ratio under the above class of variance profiles.
This calculation, which we do not include here for the sake of brevity, shows that the statistical-to-computational gap in this model of sparse PCA remains for any variance profile as above (specifically, for $B$ in the constant in \eqref{eq:uniform-profile-bounds}, there exists a $p_0 = p_0(B)$ such that, for any variance profile satisfying Assumption~\ref{assumptionA} and the condition in \eqref{eq:uniform-profile-bounds}, the same result described above holds for all $\beta \in (\beta_*^{\prime}, \beta_*)$ for some $\beta_*^{\prime} = \beta_*^{\prime}(p, B) < \beta_*$).

\subsubsection{Channel Universality}

Lastly, we remark that our results may be combined with a mild generalization of the universality results in \cite{kunisky2025low}.
That work considered more general models of the observation $Y$, where, in the context of Definition~\ref{def:inhom-wigner},
\[ Y_{ij} \sim \gamma_{X_{ij}} \]
for some ``channel'' of probability measures $(\gamma_x)_{x \in \R}$ through which one observes the signal $X$.
While that work took this family to not depend on the indices $(i, j)$, it is straightforward to extend those results to families
\[ Y_{ij} \sim \gamma^{(i, j)}_{X_{ij}} \]
provided the $\gamma^{(i, j)}_x$ satisfy some uniform regularity over different $(i, j)$.
See Remark~5 of \cite{kunisky2025low} for more discussion.

Together with our calculations, such a generalization of the main bounds of \cite{kunisky2025low} gives lower bounds against low-degree polynomials for additional inhomogeneous observation models beyond ones with additive Gaussian noise, requiring only that the $Y_{ij}$ are observed independently (and some technical regularity conditions on the entrywise channels $\gamma^{(i, j)}_x$).
Per the calculations in \cite{kunisky2025low}, the role of $\Phi = \Phi^{(N)}$, the reciprocal variance profile in our setting, would then be played by the matrix formed by the \emph{Fisher information} of each family $(\gamma^{(i, j)}_x)_{x \in \R}$ at $x = 0$.

For instance, this approach applies to inhomogeneous non-Gaussian additive noise or to the degree-corrected stochastic block model discussed in \cite{GKKZ-2025-InhomogeneousSpikedWigner}.
We remark also that, for that model and others, \cite{GKKZ-2025-InhomogeneousSpikedWigner} study analogous universality principles for information-theoretic quantities like the mutual information between the signal $X$ and the observation $Y$, while those we propose here are for computational ones, namely the low-degree $\chi^2$-divergences.

\subsection{Related Work}
\label{sec:related}

\paragraph{Spiked matrix models}
The literature on spiked matrix models in general has grown very deep since their introduction to statistics by Johnstone \cite{Johnstone-2001-LargestEigenvaluePCA}.
In particular, we have not mentioned many important related directions such as spiked models of rectangular or covariance matrices, for which analogs of our setting would also be sensible to analyze.
A survey of many of these models from a statistical point of view is given by \cite{JP-2018-PCASurvey}, and useful references for their mathematical analysis are \cite{CDM-2016-DeformedRandomMatricesFreeProbability,Capitaine-2017-HDRRandomMatrices}.
There is a long and ongoing line of work on the algorithmic tractability of such models in the homogeneous case; some relevant results, also cited above, include \cite{deshpande2015asymptotic,KXZ-2016-MutualInformationMatrix,BDMKLZ-2016-MutualInformationReplica,lesieur2017constrained,barbier2018rank,PWBM-2018-PCAI,LM-2019-LimitsLowRankEstimation,AKJ-2018-SpikedWigner}.

\paragraph{Inhomogeneous spiked matrix models}
A rectangular version of the model we describe in Definition~\ref{def:inhom-wigner} was studied by \cite{BR-2020-InfoTheoryMultiviewSpikedMatrix}, while the block-structured case of our model was studied by \cite{BR-2022-InfoTheoryInhomogeneousSpikedMatrix}.
The same block-structured model is proposed in \cite{ACCM-2022-PhysicsMultiChannelSpikedWigner}, following prior work on similar questions formulated in the language of statistical physics in \cite{ACCM-2021-MultiSpeciesMeanFieldSpinGlass,ACCM-2021-DeepBoltzmannNishimori}.
The spectral algorithm using $\sH(Y)$ was derived by \cite{pak2024optimal}, who also conjectured a version of Theorem~\ref{thm:bbp-inhom}.
Information-theoretic lower bounds and a derivation of the conjectural algorithmic threshold $\beta_*$ were carried out by \cite{GKKZ-2025-InhomogeneousSpikedWigner}, again only pertaining to block-structured models (and also requiring some further structural assumptions).
Theorem~\ref{thm:bbp-inhom} on the BBP transition for $\sH(Y)$ was proved concurrently by \cite{BCSVH-2024-MatrixConcentrationFreeProbability2,MKK-2024-OptimalPCABlockSpikedMatrix}; a similar result for the special case of $n = 2$ blocks was obtained previously in \cite{lee2024phase}.
To the best of our knowledge, there does not yet exist a version of Theorem~\ref{thm:bbp-inhom} for a general class of non-block-structured variance profiles.
The closest known seem to be the results of \cite{bigot2021freeness}, which concern $Y$ rather than $\sH(Y)$ and do not give closed forms for the outlier eigenvalue or the correlation between the signal vector $x$ and $v_1(Y)$.

\paragraph{Inhomogeneity in random matrix theory}
Inhomogeneous variance profiles have been studied more thoroughly in random matrix theory without the statistical setting we work in here.
A more general family of variance profile that is common in those works is to have $\Delta_{ij} = f(i/N, j/N)$ for some symmetric $f: [0, 1]^2 \to \R_{\geq 0}$, usually subject to various further regularity assumptions.
Block-structured models can then be expressed as step functions, but ``smoother'' variance profiles can also be expressed by continuous $f$.
We give a numerical study of BBP transitions for an example of such a model in Section~\ref{sec:numerical}.
Prior work has established characterizations of the limiting empirical spectral distribution for matrices with independent entries under certain such variance profiles as well as various local properties of the eigenvalues; see, e.g.,
\cite{shlyakhtenko1996random,ajanki2017universality,ajanki2019quadratic,zhu2020graphon}.
Completely general variance profiles have also appeared in various results in random matrix theory.
The most common instance of these is the case of \emph{generalized Wigner matrices}, which have the boundedness condition \eqref{eq:uniform-profile-bounds} and $\sum_{j = 1}^N \Delta_{ij} = N$ for all $i \in [N]$.
Such matrices are known to share many structural properties, including the limiting semicircle law for the empirical spectral distribution, with the case $\Delta_{ij} \equiv 1$ of i.i.d.\ matrices (see, e.g., \cite{ajanki2019quadratic} or Corollary 2.11 of \cite{BBVH-2021-MatrixConcentrationFreeProbability} for a remarkable further generalization).
Outside of this special class, most results seek only to control only some partial information about the spectrum like its moments or the spectral norm \cite{BVH-2016-NonasymptoticBoundsRandomMatrices,VanHandel-2017-StructuredRandomMatrices,LvHY-2018-NonhomogeneousRandomMatrices,BBVH-2021-MatrixConcentrationFreeProbability}.

\paragraph{Low-degree polynomial algorithms}
Studying low-degree polynomials as algorithms for hypothesis testing problems originates in certain technical considerations in the study of sum-of-squares semidefinite programming relaxations \cite{BHKKMP-2019-PlantedClique}.
Such algorithms were then first studied as a function class of independent interest by works including \cite{HKPRSS-2017-SOSSpectral,HS-2017-BayesianEstimation,Hopkins-2018-Thesis}.
In recent years, their analysis has become a popular approach to giving evidence of computational hardness of various problems.
We do not attempt to survey the full literature here, but the surveys \cite{kunisky2019notes,Wein-2025-LowDegreeSurvey} give a useful overview.

\section{Notation}

\paragraph{Basic notions}
We denote by $\R$ the set of real numbers, and its non-negative and strictly positive subsets by $\R_{\ge 0} \colonequals \{x \in \R : x \geq 0\}$ and $\R_{> 0} \colonequals \{x \in \R : x > 0\}$. 
For a positive integer $n$, we write $[n] = \{1,2,\dots,n\}$. For even $2n$, we will sometimes identify $[2n]$ with the paired index set $\{1,1',2,2',\dots,n,n'\}$, where the correspondence $i \leftrightarrow i'$ indicates paired elements of $[2n]$ (see Section~\ref{sec:graph-sums}).
The notation $1$ may denote either the scalar $1$ or the all-ones vector of appropriate dimension, depending on the context. For an event $E$, we denote by $\bm{1}\{E\}$ the indicator function of $E$, i.e.,
\[
\bm{1}\{E\} \colonequals
\begin{cases}
1 & \text{if } E \text{ occurs},\\
0 & \text{otherwise.}
\end{cases}
\]

\paragraph{Linear algebra}
For any set $S \subseteq \R \cup \{+\infty\}$, we write $S^{m \times n}$ for the collection of all $m \times n$ matrices with entries in $S$. When $m = n$, the subset of symmetric matrices in $S^{n \times n}$ is denoted $S^{n \times n}_{\sym}$.
For a symmetric matrix $A \in \R^{n \times n}_{\sym}$, we write $\lambda_1(A) \geq \cdots \geq \lambda_n(A)$ for the sorted (real) eigenvalues of $A$.
The operation $\vec(A)$ denotes the vectorization of $A$, obtained by stacking its columns into a single vector in $\R^{n^2}$. 
The trace of a square matrix $A$ is denoted by $\Tr(A) = \sum_{i=1}^n A_{ii}$.
For a matrix $A \in \R^{m \times n}$, we denote by $\|A\|$ its spectral norm, and by $\|A\|_{\ell_{\infty}}$ its entrywise $\ell_\infty$ norm, defined as
\[
\|A\| \colonequals \sup_{\|x\| = 1} \|A x\|, \quad \|A\|_{\ell_{\infty}} \colonequals \sup_{i,j} \left|A_{i,j}\right|.
\]
For a vector $x \in \R^n$, we write $\|x\| = (\sum_{i=1}^n x_i^2)^{1/2}$ for its Euclidean norm, and $\|x\|_{\ell_\infty} = \max_{1 \le i \le n} |x_i|$ for its $\ell_\infty$-norm. The standard Euclidean inner product between $x, y \in \R^n$ is denoted by $\langle x, y \rangle = x^\top y$. The standard Euclidean or Frobenius inner product between matrices $A, B \in \R^{m \times n}$ is denoted by $\langle A, B \rangle = \Tr(A^{\top} B)$.
The operation $\diag(\cdot)$ is defined as follows:
\begin{itemize}
    \item when applied to a vector $v \in \R^n$, $\diag(v)$ denotes the diagonal matrix with diagonal entries $(v_1,\dots,v_n)$;
\item when applied to a matrix $A \in \R^{n \times n}$, $\diag(A)$ denotes the diagonal matrix that retains the diagonal entries of $A$ while setting all off-diagonal entries to zero.
\end{itemize}
For element-wise operations on matrices of the same dimension, we write $A \odot B$ for their Hadamard product, defined by $(A \odot B)_{ij} = A_{ij} B_{ij}$, and $A^{\odot -1}$ for the Hadamard inverse, whenever all entries of $A$ are nonzero, defined by $(A^{\odot -1})_{ij} = A_{ij}^{-1}$.
$A \otimes B$ denotes the Kronecker product between matrices $A$ and $B$, defined by
\[
A \otimes B = [A_{ij} B]_{i,j},
\]
yielding a block matrix of size $(m p) \times (n q)$ when $A \in \R^{m \times n}$ and $B \in \R^{p \times q}$.

\paragraph{Asymptotics}
We use the standard asymptotic notations $O(\cdot), o(\cdot), \Omega(\cdot), \omega(\cdot)$, and $\Theta(\cdot)$ in their usual sense, always referring to the limit $N \to \infty$. 

\paragraph{Probability}
We write $\N(\mu, \Sigma)$ for the multivariate Gaussian distribution with mean $\mu$ and covariance $\Sigma$, 
$\Pois(\lambda)$ for the Poisson distribution with rate $\lambda$, 
$\chi^2(k)$ for the chi-squared distribution with $k$ degrees of freedom, 
and $\Gamma(\alpha, \beta)$ for the Gamma distribution with shape parameter $\alpha$ and scale parameter $\beta$.
Throughout, the notation $X_N \pto X$ denotes convergence in probability, and $X_N \dto X$ denotes convergence in distribution. We write $\chi^2(\P \mid \Q)$ for the chi-squared divergence between two probability measures $\P$ and $\Q$, defined as
\[
\chi^2(\P \mid \Q) = \int \left(\frac{d\P}{d\Q} - 1\right)^2 d\Q,
\]
whenever $\P$ is absolutely continuous with respect to $\Q$.

\section{Preliminaries}

\subsection{Linear Algebra}
\label{sec:linalg}

We will apply the following result to the reciprocal variance profile matrix $\Phi^{(N)}$ throughout without mention.
\begin{theorem}[Non-negative Perron-Frobenius]
    \label{thm:perron}
    Let $A \in \R^{N \times N}_{\sym}$ have non-negative entries $A_{ij} \geq 0$.
    Then, $\|A\| = \lambda_1(A) \geq 0$, i.e., for all $i = 2, \dots, N$, we have $|\lambda_i(A)| \leq \lambda_1(A)$.
\end{theorem}

The following general results about block-structured matrix sequences will be useful.
Suppose that $A^{(N)} \in \R^{N \times N}_{\sym}$ is block-structured with parameters $(\bar{A}, \rho_1, \dots, \rho_n)$.
For a given $N \geq 1$, recall that this means there is a $q: [N] \to [n]$ such that $A^{(N)}_{ij} = \bar{A}^{(N)}_{q(i)q(j)}$.

\begin{proposition}
    \label{prop:block-structured-spectrum}
    Let $S_i \colonequals q^{-1}(i)$ for $i = 1, \dots, n$, let $s_i \colonequals \sqrt{|S_i|}$, and $s = s^{(N)}$ be the vector of the $s_i$.
    Consider the matrix
    \[ B = B^{(N)} = (ss^{\top}) \circ \bar{A} \in \R^{n \times n}_{\sym}. \]
    Then, the following hold:
    \begin{enumerate}
    \item The non-zero eigenvalues of $A^{(N)}$ are the same as those of $B$.
    \item Suppose $v = (v_1, \dots, v_n)^{\top} \in \R^n$ is an unit eigenvector of $B$ with eigenvalue $\nu$.
    Then, $w \in \R^N$ with entries
    \[ w_i = \frac{v_{q(i)}}{s_{q(i)}} \]
    is a unit eigenvector of $A^{(N)}$ with the same eigenvalue $\nu$.
    \item Let $\rho$ be the vector of the $\rho_i$, and $\sqrt{\rho}$ be its entrywise square root.
    Let
    \[ \widetilde{B} \colonequals (\sqrt{\rho}\sqrt{\rho}^{\top}) \circ \bar{A} = \diag(\rho)^{1/2}\bar{A}\,\diag(\rho)^{1/2}. \]
    Then,
    \[ \lim_{N \to \infty} \frac{\lambda_1(A^{(N)})}{N} = \lambda_1(\widetilde{B}). \]
    \end{enumerate}
\end{proposition}
\begin{proof}
    Let $R \in \{0,1\}^{N \times n}$ denote the block indicator matrix, defined by
\[
R_{j,i} = \mathbf{1}\{q(j) = i\}, \quad \text{for } j \in [N] \text{ and } i \in [n].
\]
By construction we have $R^\top R = G \colonequals \operatorname{diag}(\left|S_1\right|,\ldots,\left|S_n\right|)$, and we may express the block structure of $A^{(N)}$ by writing
$A^{(N)} = R\bar{A} R^\top$. It follows from the construction that
$$B \colonequals (ss^\top) \circ \bar{A} = G^{1/2} \bar{A} G^{1/2}. $$

\noindent
\emph{Part 1: Equality of non-zero spectra.}
Recall that for any pair of matrices $X,Y$ with appropriate dimensions, the non-zero eigenvalues of $XY$ and $YX$ coincide. Applying this observation, we infer that $A^{(N)} = R \bar{A} R^\top$ has the same non-zero eigenvalues as  $\bar{A} R^\top R = \bar{A}G = \bar{A} G^{1/2} G^{1/2}$,
which in turn has the same non-zero eigenvalues as $G^{1/2} \bar{A} G^{1/2} = B$.
Hence, $A^{(N)}$ and $B$ share the same non-zero spectrum.

\vspace{1em}

\noindent
\emph{Part 2: Relation between eigenvectors.}
Let $w$ be a unit eigenvector of $A^{(N)}$ associated with a non-zero eigenvalue $\mu$, so that $A^{(N)} w = \mu w$. Since $A^{(N)} = R \bar{A} R^\top$, we obtain $R(\bar{A} R^\top w) = \mu w$,
which shows that $w$ must lie in the column space of $R$.
That is, there exists $x \in \mathbb{R}^n$ with $w = Rx$. Substituting, $R^\top R \bar{A} R^\top R x = \mu R^\top R x$, which simplifies to $G \bar{A} Gx = \mu Gx$.
Equivalently,
\[
G^{1/2} \bar{A} G^{1/2} (G^{1/2} x) = \mu (G^{1/2} x),
\]
which is $Bv = \mu v$, where $v \colonequals G^{1/2}x$.
Thus $v$ is an unit eigenvector of $B$ with eigenvalue $\mu$, and conversely, $w = Rx = R G^{-1/2} v$. In coordinates this reads
\[
w_i = \frac{v_{q(i)}}{c_{q(i)}} \quad \text{for all } i \in [N],
\]
as claimed.

\vspace{1em}

\noindent
\emph{Part 3: Asymptotic spectral limit.}
Recall that the block structure assumption implies that $|S_i|/N \to \rho_i$ as $N \to \infty$ for each block $i \in [n]$. Then, entrywise,
\[
\frac{1}{N} B^{(N)} \to \widetilde{B},
\]
where $\widetilde{B}$ is the deterministic matrix with entries $\widetilde{B}_{ij} = \sqrt{\rho_i} \sqrt{\rho_j} \bar{A}_{ij}$. 
Since these matrices have a constant dimension $n$ and converge entrywise, we equivalently have
\[
\lim_{N \to \infty} \left\| \frac{1}{N} B^{(N)} - \widetilde{B} \right\| = 0,
\]
which implies
\[
\lim_{N \to \infty} \frac{\lambda_1(B^{(N)})}{N} = \lambda_1(\widetilde{B}).
\]
By Part 1, the same conclusion holds for $A^{(N)}$, namely
\[
\lim_{N \to \infty} \frac{\lambda_1(A^{(N)})}{N} = \lambda_1(\widetilde{B}),
\]
completing the proof.
\end{proof}

\subsection{Combinatorics}
\label{sec:combinatorics}

\begin{definition}[Partitions]
    A \emph{partition} of a finite set $\sI$ is a set of sets $\pi = \{S_1, \dots, S_m\}$ where each $S_i \subseteq \sI$, each $S_i$ is non-empty, the $S_i$ are pairwise disjoint, and the disjoint union of the $S_i$ equals $\sI$.
    We call the $S_i$ the \emph{blocks} of $\pi$, and $\sI$ the \emph{base set}.
    We write $b_{\pi}: \sI \to \pi$ for the function that maps $x \in \sI$ to the (unique) block $S_i \in \pi$ such that $x \in S_i$.

    We write $\Part(\sI)$ for the set of all partitions of $\sI$, $\Part_m(\sI)$ for the set of all partitions into $m$ blocks, and $\EvenPart(\sI)$ and $\EvenPart_m(\sI)$ for the same notions but only including \emph{even partitions}, ones where $|S_i|$ is even for all $i \in [m]$.
    We write $\Match(\sI)$ for the partitions all of whose blocks have size exactly 2 (which can be viewed as perfect matchings of $\sI$).
\end{definition}
\noindent
We remark, since it will often come up in our calculations, that when $|\sI| = 2d$ then, by definition, $\Match(\sI) = \EvenPart_d(\sI)$.

\begin{definition}[Partial ordering of partitions] Two partitions $\rho, \pi$ on the same base set have the relation $\rho \preceq \pi$ if there exists a function $q: \pi \to \rho$ such that $S \subset q(S)$ for all $S \in \pi$.
In this case, we say that $\pi$ is \emph{finer than} or \emph{a refinement of} $\rho$, and conversely that $\rho$ is \emph{coarser than} or \emph{a coarsening of} $\pi$.
The relation $\preceq$ defines a partial ordering on the sets $\Part(\sI)$ and $\EvenPart(\sI)$.
\end{definition}

\begin{definition}[Stirling numbers]
    The \emph{Stirling numbers of the second kind} are
    \[ S_2(d, m) \colonequals |\Part_m([d])|. \]
\end{definition}

\begin{definition}[Touchard polynomials]
    \label{def:touchard}
    The \emph{Touchard polynomials} are the polynomials
    \[ T_d(x) \colonequals \sum_{m = 1}^d S_2(d, m) x^m. \]
\end{definition}

\noindent
As we will see below, the Touchard polynomials also have a probabilistic interpretation that will be useful to us.

\subsection{Probability}

To prove our lower bounds, we will use that certain random variables are asymptotically Gaussian by the central limit theorem.
The following multidimensional sufficient condition will be useful in some cases.

\begin{theorem}[Cram\'{e}r-Wold central limit theorem, Corollary 4.5 of \cite{Kallenberg-1997-FoundationsModernProbability}]
    \label{thm:cramer-wold}
    Let $X^{(k)} \in \R^n$ be a sequence of random vectors and let $X$ be another random vector.
    Then, the $X^{(k)}$ converge in distribution to $X$ as $k \to \infty$ if and only if $\langle c, X^{(k)}\rangle = \sum_{i = 1}^n c_i X^{(k)}_i$ converges in distribution to $\langle c, X \rangle = \sum_{i = 1}^n c_iX_i$ for all $c \in \R^n$.
    
    In particular, suppose that $\langle c, X^{(k)}\rangle$ converges in distribution to $\sN(0, \|c\|^2)$ for all $c \in \R^n$.
    Then, $X^{(k)}$ converges in distribution to $\sN(0, I_n)$.
\end{theorem}

To use these results, we will also want to give lower bounds on the moments of Gaussian quadratic forms.
The following result follows easily from Corollary~1 of \cite{bao2006expectation}.
\begin{proposition}
\label{prop:baopaper}
    Let $A \in \R^{n \times n}_{\sym}$ satisfy that $\Tr(A^k) \geq 0$ for all $k \geq 1$.
    For instance, we can either have $A \succeq 0$ or $A$ having only non-negative entries.
    Let $y \in \R^n$ be a random vector with $y_i \sim \sN(0, 1)$ i.i.d.
    Then, we have for all $d \geq 1$ that
    \begin{align*}
        \E[(y^{\top}Ay)^{d}] \geq 2^{d-1} (d-1)! \,\Tr(A^d).
    \end{align*}
\end{proposition}

Our calculations will involve various sums over partitions which will lead to evaluations of the Touchard polynomials, defined above.
We use the following probabilistic interpretation to derive sufficiently tight bounds on them.
\begin{proposition}
    For any $\lambda > 0$ and $d \geq 1$,
    \[ \Ex_{X \sim \Pois(\lambda)} X^d = T_d(\lambda). \]
\end{proposition}

We will use the following bound, which in fact, as discussed in the reference, holds for all ``sub-Poissonian'' random variables (a notion similar to subgaussian and subexponential).
\begin{proposition}[\cite{ahle2022sharp}]
\label{prop:ahle}
    For all $d \geq 1$ and $\lambda > 0$, we have
    \[ \Ex_{X \sim \Pois(\lambda)} \left(\frac{X}{\lambda}\right)^d \leq \left(\frac{\frac{d}{\lambda}}{\log(1 + \frac{d}{\lambda})}\right)^d. \]
\end{proposition}

\begin{corollary}
\label{cor:touchard-bound}
    For all $d \geq 1$ and $\lambda > 0$, we have
    \[ T_d(\lambda) = \Ex_{X \sim \Pois(\lambda)} X^d \leq \lambda^d \exp\left(\frac{d^2}{\lambda}\right). \]
\end{corollary}
\begin{proof}
    From Proposition~\ref{prop:ahle}, to reach this claim it suffices to show the scalar inequality $x / \log(1 + x) \leq \exp(x)$ for all $x > 0$ (which we apply with $x = d / \lambda$).
    This result follows since the function $f(x) \colonequals \log(1 + x) - x\exp(-x)$ has $f(0) = 0$ and 
    \[ f^{\prime}(x) = \frac{1}{1 + x} + x\exp(-x) - \exp(-x) = \frac{\exp(x) - 1 + x^2}{(1 + x)\exp(x)} \geq 0 \]
    for all $x \geq 0$.
\end{proof}

Lastly, we will need some properties of subgaussian random variables, as well as a few variants thereof.
The following is the standard definition of subgaussianity.
\begin{definition}[Subgaussianity]
    We say that a centered random variable $X$ is \emph{$\sigma^2$-subgaussian} if $\E \exp(tX) \leq \exp(\frac{\sigma^2}{2}t^2)$ for all $t \in \R$.
    We say that a centered random vector $x$ is $\sigma^2$-subgaussian if $\langle x, v \rangle$ is $\sigma^2$-subgaussian for all $\|v\| = 1$.
\end{definition}
\noindent
For some of our applications, the following version will also be useful.
\begin{definition}[Moment subgaussianity]
    \label{def:moment-subgaussian}
    We say that a symmetric random variable $X$ is \emph{$\sigma^2$-moment-subgaussian} if, for all $k \geq 1$, $$\E X^{2k} \leq \sigma^{2k}(2k-1)!! = \Ex_{Z \sim \sN(0, \sigma^2)} Z^{2k}.$$
    We say that a symmetric random vector $x$ is $\sigma^2$-moment-subgaussian if $\langle x, v \rangle$ is $\sigma^2$-moment-subgaussian for all $\|v\| = 1$.
\end{definition}

These definitions are equivalent up to constants in $\sigma^2$; however, for some purposes we will need to be careful about those constants.
\begin{proposition}
    \label{prop:subgauss-moments}
    The following hold for a symmetric random variable $X$:
    \begin{enumerate}
        \item If $X$ is $\sigma^2$-subgaussian, then $X$ is $4\sigma^2$-moment-subgaussian.
        \item If $X$ is $\sigma^2$-moment-subgaussian, then $X$ is $\sigma^2$-subgaussian.
    \end{enumerate}
\end{proposition}
\begin{proof}
    For the second implication, we may consider the moment generating function directly: for $G \sim \sN(0, \sigma^2)$,
    \begin{align*}
    \E \exp(tX) 
    &= \sum_{k = 0}^{\infty} \frac{t^k}{k!} \E X^k \\
    &= \sum_{k = 0}^{\infty} \frac{t^{2k}}{(2k)!} \E X^{2k} \tag{symmetry of $X$} \\ 
    &\leq \sum_{k = 0}^{\infty} \frac{t^{2k}}{(2k)!} \E G^{2k} \\
    &= \E \exp(tG) \\
    &= \exp\left(\frac{\sigma^2}{2}t^2\right).
    \end{align*}
    For the first implication, by Lemma~1.4 of \cite{Rigollet-2017-HighDimensionalStatistics}, we have
    \begin{align*}
        \E X^{2k}
        &\leq (2\sigma^2)^k \cdot 2 \cdot k! \\
        &\leq (4\sigma^2)^k \cdot k! \\
        &\leq (4\sigma^2)^k \cdot (2k - 1)!!,
    \end{align*}
    giving the result.
\end{proof}

\begin{proposition}
    \label{prop:subgauss-tau}
    If $X$ is $\sigma^2$-subgaussian and has $\E X = 0$ and $\E X^2 = 1$, then there exists a $\tau > 0$ such that, for all $k \geq 1$,
    \[ \E X^{2k} \leq \tau^{2k - 2} (2k - 1)!!. \]
\end{proposition}
\begin{proof}
    By Proposition~\ref{prop:subgauss-moments}, $X$ is $4\sigma^2$-moment-subgaussian.
    Take $\tau^2 = \max\{1, (4\sigma^2)^3\}$, so that $\E X^{2k} \leq \tau^{2k/3} (2k - 1)!!$.
    Since $2k / 3 \leq 2k - 2$ for all $k \geq 2$, the result follows.
\end{proof}

We will also use the following result about moment-subgaussian vectors.
\begin{proposition}
\label{prop:moment-bound}
    If a random vector $x \in \R^k$ is $1$-moment-subgaussian then, for any $\delta \in (0, 1)$ and $\gamma \in (0, \frac{1 - \delta}{2})$,
    \begin{align*}
    \E\left[\exp\left(\gamma \|x \|^2\right)\right] &\leq C(\delta,k) \Ex_{Z \sim \sN(0, 1)}\left[\exp\left(\frac{\gamma}{1-\delta} Z^2\right)\right] \\
    &=C(\delta,k) \frac{1}{\sqrt{1-\frac{2\gamma}{1-\delta}}},
    \end{align*}
    where $C(\delta, k)$ is a constant depending only on $\delta$ and $k$.
    In particular, we have
    \[ \E\left[\exp\left(\gamma \|x \|^2\right)\right] < \infty \]
    for all $\gamma < \frac{1}{2}$.
\end{proposition}

\begin{proof}
Let $\mathcal{N} \subset \R^k$ be a $\delta$-net of the unit sphere of $\R^k$.
Standard bounds give
\[ \|x\| \leq \frac1{1-\delta}\max_{u \in \mathcal{N}} \langle u, \bm{x} \rangle. \]
Therefore, we may bound the moments by
\begin{align*}
    \E\left[\| x \|^{2k}\right] &\leq \left(\frac1{1-\delta}\right)^{2k} \E\left[\max_{u \in \mathcal{N}} \langle u, x \rangle^{2k}\right] \\
     &\leq \left(\frac1{1-\delta}\right)^{2k} \sum_{u \in \mathcal{N} } \E\left[ \langle u, x \rangle^{2k}\right] \\
     &\leq \left(\frac1{1-\delta}\right)^{2k} \sum_{u \in \mathcal{N} } \E[ Z^{2k}] \\
     &= C(\delta, k) \left(\frac1{1-\delta}\right)^{2k} \E[Z^{2k}]
\end{align*}
where $C(\delta,k) = |\mathcal{N}|$ is a constant depending only on $\delta, k$. 
The result then follows by expanding $\exp(\gamma\|x\|^2)$ in a Taylor series.
\end{proof}

Now we will state some definitions and propositions proposed in \cite{bandeira2021spectral} concerning a ``local'' version of subgaussianity, formalizing ideas appearing previously in \cite{PWBM-2018-PCAI}.
\begin{definition}
    A random vector $x \in \R^k$ is $\varepsilon$-locally $c$-subgaussian if, for any vector $v \in \R^k$ with $\|v \| \leq \varepsilon$, $$\E \exp\left(\langle v, x\rangle\right) \leq \exp\left(\frac{c}2\|v\|^2\right).$$
\end{definition}

\begin{proposition}
\label{prop:local-subgaussian}
Suppose $x$ is $\varepsilon$-locally $c$-subgaussian. For a (non-random) scalar $\alpha \neq 0$, $\alpha x$ is $\varepsilon/|\alpha|$-locally $c \alpha^2$-subgaussian. Also, the sum of $n$ independent copies of $x$ is $\varepsilon$-locally $cn$-subgaussian.
\end{proposition}
\begin{proposition}
\label{prop:tail-bound}
    If $x \in \R^k$ is $\varepsilon$-locally $c$-subgaussian, then for any $\delta > 0$ and any $0 \leq t \leq  \varepsilon c/(1 - \delta)$, $$\P\left[\left\|x\right\| \geq t\right] \leq C(\delta, k) \exp\left(-\frac1{2c}(1-\delta)^2 t^2\right)$$ where $C(\delta, k)$ is a constant depending only on $\delta$ and $k$.
\end{proposition}

\begin{proposition}
\label{prop:expectation-bound}
    Let $\delta > 0$. If a random vector $x \in \R^k$ is $\varepsilon$-locally $c$-subgaussian with $c < (1-\delta)^2/2$ then $$\E\left[\bm{1}{\left\{\|x\|\leq \varepsilon c/(1-\delta)\right\}} \cdot \exp(\|x\|^2)\right]\leq 1 + \frac{C(\delta,k)}{(1-\delta)^2/2c - 1}$$
    where $C(\delta, k)$ is a constant depending only on $\delta$ and $k$.
\end{proposition}

\subsection{Calculus}

The following is a standard series evaluation, a special case of Newton's generalized binomial theorem.
\begin{proposition}[Hypergeometric sum]
\label{prop:hypergeometric_sum}
For $|a| < 1$ and $c > 0$,
\[ \sum_{d=0}^\infty \frac{a^d}{d!} \frac{\Gamma(d + c)}{\Gamma(c)} = (1-a)^{-c}. \]
\end{proposition}
\noindent
We will also need to bound the output of the above evaluation, for which we apply the following useful bound.
\begin{proposition}
\label{prop:log-bound}
For $a \in (0, 1)$ and $c > 0$,
    \[ \exp(ca) \leq (1 - a)^{-c} \leq \exp\left(\frac{1 - \frac{a}{2}}{1 - a} \cdot ca\right). \]
\end{proposition} 
\noindent
This result may be viewed as saying that the simple lower bound above is actually tight up to constant factors in the exponent, provided that $a$ is not too close to 1.
\begin{proof}
    The lower bound follows immediately from the bound $1 - a \leq \exp(-a)$.
    For the upper bound, we take the Taylor expansion of the logarithm,
    \begin{align*}
        (1 - a)^{-c} 
        &= \exp\left(-c\log\left(1-a\right)\right) \\
        &= \exp\left(c\sum_{p=1}^\infty \frac{a^p}{p}\right) \\
        &\leq \exp\left(c\left(a + a^2 \sum_{p=0}^\infty\frac{a^p}{2}\right) \right) \\
        &= \exp\left(c\left(a + \frac{1}{2} \cdot \frac{a^2}{1-a}\right) \right) \\
        &= \exp\left(ca\left(1 + \frac{1}{2} \cdot\frac{a}{1-a}\right) \right) \\
        &= \exp\left(ca \cdot \frac{1}{2} \cdot \frac{2 - a}{1-a}\right),
    \end{align*}
    giving the result.
\end{proof}

\subsection{Low-Degree Polynomial Algorithms}
\label{subsection:low-degree}

We introduce some tools, by now standard in the literature, for showing that low-degree polynomials cannot strongly separate sequences of probability measures.
See the surveys \cite{kunisky2019notes,Wein-2025-LowDegreeSurvey} for many more details.
The technique we will use revolves around the following quantities.

\begin{definition}[Low-degree $\chi^2$-divergence]
\label{def:low-deg-div}
Consider a pair of probability measures $\P, \Q$ over $\R^{N \times N}_{\sym}$ (say in the setting of Definition \ref{def:inhom-wigner}). The \emph{degree-$D$ $\chi^2$-divergence} is given by
\begin{align*}
    \chi^2_{\leq D}(\P \mid \Q) 
    &\colonequals \sup_{f \in \R[Y]_{\leq D} \setminus \{0\}} \frac{(\mathbb{E}_{\P}[f(Y)])^2}{\mathbb{E}_{\mathbb{Q}}[f(Y)^2]} - 1 \\
    &= \sup_{\substack{f \in \R[Y]_{\leq D} \setminus \{0\} \\ \mathbb{E}_{\Q} f(Y) = 0}}\frac{(\mathbb{E}_{\P}[f(Y)])^2}{\mathbb{E}_{\mathbb{Q}}[f(Y)^2]}.
\end{align*}
\end{definition}
\noindent
If we remove the restriction to low-degree polynomials, then this expression gives the classical $\chi^2$-divergence $\chi^2(\P \mid \Q)$ (provided that $\frac{d\P}{d\Q} \in L^2(\Q)$).
The kind of analysis we carry out is more commonly formulated in terms of the low-degree ``advantage'' $1 + \chi^2_{\leq D}(\P \mid \Q)$, but we follow the above notation (also appearing, for instance, in \cite{COGHWZ-2022-PhaseTransitionsGroupTesting}) to emphasize the analogy with the $\chi^2$ divergence, which appears in our statistical lower bounds.

The low-degree and classical $\chi^2$-divergences govern strong separation by low-degree polynomials and by arbitrary functions in the following way.
\begin{proposition}[Proposition 6.2 of \cite{BAHSWZ-2022-FranzParisiLowDegree}, Lemma 1.13 of \cite{kunisky2019notes}]
    \label{prop:chi-squared-testing}
    For a sequence of pairs of probability measures $\P_N, \Q_N$ over $\R^{N \times N}_{\sym}$, the following hold:
    \begin{enumerate}
        \item If $\chi^2_{\leq D(N)}(\P_N \mid \Q_N) = O(1)$ as $N \to \infty$, then there is no sequence of polynomials $f_N \in \R[Y]$ with $\deg(f_N) \leq D(N)$ achieving strong separation.
        \item If $\chi^2(\P_N \mid \Q_N) = O(1)$ as $N \to \infty$, then there is no sequence of (measurable) functions $f_N: \R^{N \times N}_{\sym} \to \R$ achieving strong separation, and no sequence of tests $t_N: \R^{N \times N}_{\sym} \to \{0, 1\}$ achieving strong detection.
    \end{enumerate}
\end{proposition}
\noindent
Importantly, neither converse statement is true in general, as discussed earlier in Section~\ref{sec:low-deg-div}.

\subsection{Tools for Additive Gaussian Models}
\label{sec:additive-gaussian}

It turns out that convenient and elegant closed forms for the low-degree and classical $\chi^2$-divergences are known for the special case of models with additive Gaussian noise.
Let us first state a general result for a general such model.

\begin{definition}[Additive Gaussian Noise Model]
Let $X \sim \sX = \sX_M$ for some probability measure $\sX$ over $\R^M$ and $Z \sim \sN(0, \Sigma)$ for some $\Sigma \in \R^{M \times M}_{\sym}$ with $\Sigma \succeq 0$.
The associated additive Gaussian model consists of the probability measures $\P = \P_M$ and $\Q = \Q_M$ given by:
\begin{itemize}
    \item To sample $Y \sim \P$, observe $Y = X + Z$.
    \item To sample $Y \sim \mathbb{Q}$, observe $Y = Z$.
\end{itemize}
\end{definition}

\begin{proposition}[Theorem 2.6 of \cite{kunisky2019notes}]
Suppose in the above setting that $\Sigma = I_M$.
Then,
\begin{align*}
    \chi^2(\P \mid \Q) &= \Ex_{X^1, X^2 \sim \sX}\exp\left(\left\langle X^1, X^2 \right\rangle \right) - 1, \\
    \chi^2_{\leq D}(\P \mid \Q) &= \Ex_{X^1, X^2 \sim \sX} \exp^{\leq D}\left(\left\langle X^1, X^2 \right\rangle\right) - 1,
    \end{align*}
    where we define
\[ \exp^{\leq D}(x) \colonequals \sum_{d=0}^D \frac{x^d}{d!} \]
    and $X^1, X^2$ are drawn independently from $\sX$.
\end{proposition}

Since in the more general model with arbitrary covariance $\Sigma$ we may just multiply through by $\Sigma^{-1/2}$ to obtain a model with $\Sigma = I_M$, it is easy to derive the following similar identity.
\begin{corollary}
    \label{cor:add-gauss-cov}
    For a general $\Sigma \succ 0$, in the above setting,
    \begin{align*}
    \chi^2(\P \mid \Q) &= \Ex_{X^1, X^2 \sim \sX}\exp\left(X^{1^{\top}}\Sigma^{-1} X^2 \right) - 1, \\
    \chi^2_{\leq D}(\P \mid \Q) &= \Ex_{X^1, X^2 \sim \sX} \exp^{\leq D}\left(X^{1^{\top}}\Sigma^{-1} X^2\right) - 1.
    \end{align*}
\end{corollary}

We now apply these results to our matrix models.
Since we work over symmetric matrices, it suffices to consider the entries on and above the diagonal, of which there are $M = \frac{N(N + 1)}{2}$.
But, to work in the inhomogeneous setting, we must indeed use the extra flexibility of Corollary~\ref{cor:add-gauss-cov}.
Applying this, we obtain the following, which will be the starting point of our coming calculations.
\begin{lemma}
\label{lem:low-deg-prelim}
    Let $\P_N$ and $\Q_N$ be the planted and null distributions for an inhomogeneous spiked Wigner matrix model with signal distribution $\sX_N$ and variance profile $\Delta = \Delta^{(N)}$.
    Recall that we set $\Phi \colonequals \Delta^{\odot - 1}$.
    Then, we have
    \begin{align*}
        \chi^2(\P_N \mid \Q_N)
        &= \Ex_{X^1, X^2 \sim \sX_N} \exp\left(R(X^1, X^2)\right) - 1, \\
        \chi^2_{\leq D}(\P_N \mid \Q_N) &= \Ex_{X^1, X^2 \sim \sX_N} \exp^{\leq D}(R(X^1, X^2)) - 1,
    \end{align*}
    where $R(X^1, X^2) \in \R$ is the ``overlap'' random variable, which may be written equivalently as
    \begin{align*}
        R(X^1, X^2)
        &= \sum_{1 \leq i \leq j \leq N} \Phi_{ij} X^1_{ij} X^2_{ij} \\
        &= \frac{1}{2}\sum_{i, j = 1}^N (\Phi + \Diag(\Phi))_{ij} X^1_{ij} X^2_{ij}.
    \end{align*}
    Further, suppose we are in the case of any of our Theorems in Section~\ref{sec:results}, where there is a further probability measure $\sP_N$ on $\R^{N \times \kappa}$ such that $X \sim \sX_N$ can be drawn as $X = \frac{\beta}{\sqrt{N}}VV^{\top}$ for $V \sim \sP_N$.
    Let $V^1, V^2 \sim \sP_N$ be i.i.d.\ draws, and write $x^a_1, \dots, x^a_N \in \R^{\kappa}$ for the rows and $v^a_1, \dots, v^a_{\kappa} \in \R^N$ for the columns of $V^a$ for $a \in \{1, 2\}$.
    In this case, the overlap may be rewritten as
    \begin{align*}
        R(V^1, V^2)
        &\colonequals R\left(\frac{\beta}{\sqrt{N}}V^1V^{1^{\top}}, \frac{\beta}{\sqrt{N}} V^2V^{2^{\top}}\right) \\
        &= \frac{\beta^2}{2N} \sum_{i, j = 1}^N (\Phi + \Diag(\Phi))_{ij} \langle x^1_i, x^1_j \rangle \langle x^2_i, x^2_j \rangle \\
        &= \frac{\beta^2}{2N} \sum_{a, b = 1}^{\kappa} (v^1_a \odot v^2_b)^{\top} (\Phi + \Diag(\Phi)) (v^1_a \odot v^2_b).
    \end{align*}
\end{lemma}
\noindent
We note that, conveniently, this way of expressing the full and low-degree $\chi^2$-divergences captures the possibility of censored observations by setting certain entries of $\Phi$ to zero, letting us not need to explicitly keep track of the set $\fJ$ of entries with finite variance from now on.

The following basic facts about the truncated exponential polynomials $\exp^{\leq D}(x)$ will be useful (while simple and natural-looking, the second turns out to be surprisingly non-trivial to prove).
\begin{proposition}
    \label{prop:exp-mon}
    For all $D \geq 1$, $\exp^{\leq D}(x)$ is monotonically increasing over $x \geq 0$.
\end{proposition}
\begin{proposition}[Proposition 5.3 of \cite{kunisky2025low}]
    \label{prop:exp-bounds}
    For all $D \geq 2$ even and $x \in \R$, we have
    \[ 0 < \exp^{\leq D}(x) \leq \exp^{\leq D}(|x|). \]
    Further, for all $x, y \in \R$, we have
    \[ \exp^{\leq D}(x) \exp(-100|y|) \leq \exp^{\leq D}(x + y) \leq \exp^{\leq D}(x) \exp(100|y|) \]
\end{proposition}
\noindent
While this result requires $D$ to be even, for all of our purposes it will be possible to assume this without loss of generality by just increasing or decreasing $D$ by one as needed and using that the low-degree $\chi^2$-divergence is monotone in $D$.
So, we do not explicitly verify this assumption when we apply the Proposition later.

Finally, the following is a useful general tool for controlling expectations of the form that we will be dealing with, which abstracts a style of analysis of such expressions first (to the best of our knowledge) developed by \cite{BMVVX-2018-InfoTheoretic,PWBM-2018-PCAI}.
\begin{lemma}[Lemma 2.10 of \cite{kunisky2024low}]
    \label{lem:overlap}
    Let $R_N \geq 0$ be a sequence of real bounded random variables and $D(N) \in \mathbb{N}$.
    Suppose there exists a sequence $A(N) \in \R_{\geq 0}$ such that the following conditions hold:
    \begin{enumerate}
    \item For all sufficiently large $N$,
        \[ A(N) \geq D(N) \left(2 \vee \log\left(\frac{\|R_N\|_{\infty}}{D(N)}\right)\right). \]
        (Recall that $\|R_N\|_{\infty}$ is the smallest $C > 0$ such that $R_N \leq C$ almost surely.)
    \item For some bounded $f: \R_{\geq 0} \to \R_{\geq 0}$ such that $\int_0^{\infty}f(t)\,dt < \infty$, for sufficiently large $N$, for all $t \in [0, A(N)]$,
        \[ \P[R_N \geq t] \leq f(t) \exp(-t). \]
    \end{enumerate}
    Then, we have
    \[ \limsup_{N \to \infty} \E \exp^{\leq D(N)}(R_N) < \infty. \]
\end{lemma}

\section{Block-Structured Variance Profiles: Proof of Theorem \ref{theorem:comp-lower-bound-block-constant-variance}}

Recall that Theorem \ref{theorem:comp-lower-bound-block-constant-variance} concerns the inhomogeneous spiked Wigner matrix model (Definition~\ref{def:inhom-wigner}) where the variance profiles $\Delta = \Delta^{(N)}$ are a block-structured sequence of matrices.
Consequently, the Hadamard inverses $\Phi = \Phi^{(N)}$ are also block-structured.
The Theorem includes upper bounds on $\chi^2_{\leq D}(\P_N \mid \Q_N)$ (implying computational lower bounds of the form that polynomials of degree at most $D$ cannot achieve strong separation) and upper bounds on $\chi^2(\P_N \mid \Q_N)$ (implying statistical lower bounds of the form that arbitrary functions cannot achieve strong separation).
Also, in Theorem~\ref{thm:block-low-deg-div}, we give complementary lower bounds on $\chi^2_{\leq D}(\P_N \mid \Q_N)$, implying that the analysis in the first part is tight.
We consider these three results separately in the following three sections.

\subsection{Preliminaries on Block-Structured Models}

Before continuing, let us introduce some further notation to deal with the block structure of the model.
Recall that the block structure parameters are $(\bar{\Delta}, \rho_1, \dots, \rho_n)$ such that, for each $N$, there is a function $q: [N] \to [n]$ such that $\Delta^{(N)}_{ij} = \bar{\Delta}_{q(i)q(j)}$ for all $i, j \in [N]$.
Let us write $S_a = S_{a, N} \colonequals q^{-1}(a)$ for each $a \in [n]$, so that the $S_a$ form a partition of $[N]$.
Then, the block structure assumption asks that $|S_a| / N \to \rho_a \in (0, 1)$ for each $a \in [n]$.

We have previously defined $\bar{\Phi} \colonequals \bar{\Delta}^{\odot -1}$.
Let us also define
\[ \widetilde{\Gamma} \colonequals \diag(\rho)^{1/2} \, \bar{\Phi} \, \diag(\rho)^{1/2}, \]
so that our threshold parameter is
\[ \what{\mu}_1 = \lambda_1(\widetilde{\Gamma}). \]
We also define a matrix $\Gamma = \Gamma^{(N)} \in \R^{n \times n}$, given by
\[ \Gamma \colonequals \diag(|S_1|, \dots, |S_n|)^{1/2} \, \bar{\Phi} \, \diag(|S_1|, \dots, |S_n|)^{1/2}. \]
From the block structure assumption, since $|S_a| / N \to \rho_a$, one may verify that
\begin{align*}
    \lim_{N \to \infty} \frac{1}{N}\Gamma^{(N)} &= \widetilde{\Gamma}, \\
    \lim_{N \to \infty} \frac{1}{N}\lambda_1(\Gamma^{(N)}) &= \lambda_1(\widetilde{\Gamma}) \\ &= \what{\mu}_1,
\end{align*}
where the first convergence occurs entrywise (noting that the dimension $n$ of $\Gamma^{(N)}$ does not depend on $N$) and the second convergence follows from the first.

Finally, let us define the matrix $P \in \R^{n \times N}$ such that, for all $v \in \R^N$,
\[ (Pv)_a = \frac{1}{\sqrt{|S_a|}} \sum_{i \in S_a}v_i. \]
This matrix averages vectors over blocks, with a normalization \emph{a la} the central limit theorem (whose role will become clear in the proof).
This satisfies the identity
\[ v^{\top}\Phi v = (Pv)^{\top} \Gamma (Pv) \]
for all $v \in \R^N$.

\subsection{Computational Lower Bound}
\label{section:block-constant-comp}

\begin{proof}
Let us introduce a bit of notation for constants that will appear throughout.
Since the entries of $\Phi$ all appear in the finite matrix $\bar{\Phi}$, they are uniformly bounded:
\[ \|\Phi\|_{\ell^{\infty}} \leq \|\bar{\Phi}\|_{\ell^{\infty}} \equalscolon B. \]
Also, since the distribution $\nu$ of the rows of matrices drawn from $\sP_N$ is of bounded support, when $x \sim \nu$ we may assume that almost surely
\[ \|x\| \leq C. \]

Recall that our goal is to show that
\[ \chi^2_{\leq D}(\P_N \mid \Q_N) = \E \exp^{\leq D}(R) - 1 \stackrel{\text{(?)}}{=} O(1). \]
Let us begin by proving some preliminary bounds on the overlap.
We may decompose
\[ R(V^1, V^2) = R^{(0)}(V^1, V^2) + R^{(1)}(V^1, V^2), \]
where
\begin{align*}
    R^{(0)}(V^1, V^2)
    &\colonequals \frac{\beta^2}{2N} \sum_{a, b = 1}^{\kappa} (v^1_a \odot v^2_b)^{\top} \Phi (v^1_a \odot v^2_b) \\
    &= \frac{\beta^2}{2N} \sum_{a, b = 1}^{\kappa} (P(v^1_a \odot v^2_b))^{\top} \Gamma (P(v^1_a \odot v^2_b))
\end{align*}
and
\begin{align*}
    R^{(1)}(V^1, V^2)
    &\colonequals \frac{\beta^2}{2N} \sum_{a, b = 1}^{\kappa} (v^1_a \odot v^2_b)^{\top} \Diag(\Phi) (v^1_a \odot v^2_b).
\end{align*}
We drop the arguments $(V^1, V^2)$ from $R, R^{(0)}$, and $R^{(1)}$ below for the sake of clarity.
It will turn out that $R^{(0)}$ is the ``main term'' of $R$, while $R^{(1)}$ makes a negligible contribution.

Using Propositions~\ref{prop:exp-mon} and \ref{prop:exp-bounds}, we may bound
\begin{align*}
    \chi^2_{\leq D}(\P_N \mid \Q_N)
    &\leq \E \exp^{\leq D}(R) \\
    &\leq \E \exp^{\leq D}(|R|) \\
    &\leq \E \exp^{\leq D}(|R^{(0)}| + |R^{(1)}|).
\end{align*}
Let us show first that we may effectively get rid of the $|R^{(1)}|$ term here.
We have, by our boundedness assumption on $x \sim \nu$, that we have almost surely
\begin{align*}
    |R^{(1)}|
    &\leq \frac{\beta^2}{2N} \|\Phi\|_{\ell^{\infty}} \sum_{a, b = 1}^{\kappa} \|v_a^1 \odot v_b^2\|^2 \\
    &= \frac{\beta^2}{2N} \|\Phi\|_{\ell^{\infty}} \sum_{a, b = 1}^{\kappa} \sum_{i = 1}^N (v_a^1)_i^2 (v_b^2)_i^2 \\
    &= \frac{\beta^2}{2N} \|\Phi\|_{\ell^{\infty}} \sum_{i = 1}^N \|x_i^1\|^2 \|x_i^2\|^2 \\
    &\leq \frac{\beta^2 B C^4}{2},
\end{align*}
which is just a constant.
Using Proposition~\ref{prop:exp-bounds}, we may then bound
\[ \chi^2_{\leq D}(\P_N \mid \Q_N) \leq \exp(50 \beta^2 B C^4) \cdot \E \exp^{\leq D}(|R^{(0)}|), \]
and thus we have reduced our task to just showing that
\[ \E \exp^{\leq D}(|R^{(0)}|) \stackrel{\text{(?)}}{=} O(1). \]

We now work on bounding $R^{(0)}$.
Since $\Gamma$ has non-negative entries whereby, by the Perron-Frobenius theorem (Theorem~\ref{thm:perron}) we have $\lambda_1(\Gamma) = \|\Gamma\|$, we also have
\begin{align*}
    |R^{(0)}|
    &\leq \frac{\beta^2 \lambda_1(\Gamma)}{2N} \sum_{a, b = 1}^{\kappa} \|P(v^1_a \odot v^2_b)\|^2 \\
    &= \frac{\beta^2 \lambda_1(\Gamma)}{2N} \sum_{a, b = 1}^{\kappa} \sum_{h = 1}^n \left(\frac{1}{\sqrt{|S_h|}} \sum_{i \in S_h} (v^1_a)_i (v^2_b)_i\right)^2 \\
    &= \frac{\beta^2 \lambda_1(\Gamma)}{2N} \sum_{h = 1}^n \frac{1}{|S_h|} \sum_{a, b = 1}^{\kappa} \left(\sum_{i \in S_h} (v^1_a)_i (v^2_b)_i\right)^2,
\end{align*}
where we note that the factor of $\lambda_1(\Gamma) / N$ on the outside will, in the limit $N \to \infty$, lead to the factor of $\what{\mu}_1$ that we expect.
Now, let us define the matrix $T^{(h)} \in \R^{\kappa \times \kappa}$ for each $h \in [n]$ whose entries are the ``partial inner products'' appearing here:
\[ T^{(h)}_{ab} \colonequals \sum_{i \in S_h} (v^1_a)_i (v^2_b)_i. \]
Then, we may rewrite the above bound, also separating the leading factors per the above observation, as
\begin{equation}
    \label{eq:R0-main-bound}
    |R^{(0)}| \leq \beta^2 \cdot \frac{\lambda_1(\Gamma)}{N} \cdot \frac{1}{2} \sum_{h = 1}^n \frac{\|T^{(h)}\|_F^2}{|S_h|}.
\end{equation}

We would like to apply Lemma~\ref{lem:overlap} to $|R^{(0)}|$.
Recall that the Lemma asks for two conditions: a weak but almost sure bound on this random variable, and a bound on the probability of certain large deviations.
For the former, note that we can write
\[ T^{(h)} = U^{(h, 1)^{\top}} U^{(h, 2)}, \]
where $U^{(h, \ell)} \in \R^{S_h \times \kappa}$ has entries
\[ U^{(h, \ell)}_{i a} = (v^{\ell}_a)_i. \]
Equivalently, the rows of $U^{(h, \ell)}$ are those $x^{\ell}_i$ for $i \in S_h$.
Thus, $\|U^{(h, \ell)}\|_F^2 \leq |S_h| \cdot C^2$, since each of these rows has norm at most $C$ by our assumption on the boundedness of $\nu$.
Now, we have
\[ \|T^{(h)}\|_F^2 \leq \|U^{(h, 1)}\|_F^2 \|U^{(h, 2)}\|_F^2 \leq C^4 |S_h|^2, \]
and therefore, almost surely,
\begin{align*} 
|R^{(0)}| 
&\leq \beta^2 \cdot \frac{\lambda_1(\Gamma)}{N} \cdot \frac{1}{2} \sum_{h = 1}^n \frac{\|T^{(h)}\|_F^2}{|S_h|} \\
&\leq \beta^2 C^4 \cdot \frac{\lambda_1(\Gamma)}{N} \cdot \frac{1}{2} \sum_{h = 1}^n |S_h| \\
&= \frac{1}{2}\beta^2 C^4 \cdot \lambda_1(\Gamma)
\intertext{and finally, a simple bound gives that, since $|S_i| \leq N$, we have $\lambda_1(\Gamma) \leq \lambda_1(\bar{\Phi}) N = \what{\mu}_1 N$, so}
&\leq \frac{1}{2}\beta^2 C^4 \what{\mu}_1 N
\intertext{The details of the constants here will be inconsequential, so let us just write}
&\equalscolon C^{\prime}N.
\end{align*}
Since we plan to take $D(N) \leq N / \log(N)$ (indeed, even smaller, but at least satisfying this inequality for sufficiently large $N$), the condition of Lemma~\ref{lem:overlap} requires that we take the value $A(N)$ to be at least $D(N) \log\log N \leq N \cdot \frac{\log \log N}{\log N}$, for sufficiently large $N$.
In particular, it is permissible to take $A(N) \colonequals \delta N$ for a small constant $\delta > 0$ to be chosen later.

Recall that, to apply the Lemma, we then need to bound the tail probability $\P[|R^{(0)}| \geq t]$ for all $t \in [0, \delta N]$.
For the Lemma to apply, we need a tail bound on this probability that is ``slightly better'' than $\exp(-t)$.
We first take some initial steps towards showing this.
Starting from \eqref{eq:R0-main-bound}, which we repeat below, we have
\begin{align*}
    |R^{(0)}| 
    &\leq \beta^2 \cdot \frac{\lambda_1(\Gamma)}{N} \cdot \frac{1}{2} \sum_{h = 1}^n \frac{\|T^{(h)}\|_F^2}{|S_h|}
    \intertext{Now, as we noted earlier, as $N \to \infty$ we have $\lambda_1(\Gamma) / N \to \what{\mu}_1$, and we have also assumed that $\beta^2 < 1 / \what{\mu}_1$. Thus, there exists an $\epsilon > 0$ such that, for all sufficiently large $N$, we have}
    &\leq (1 - \epsilon) \cdot \frac{1}{2} \sum_{h = 1}^n \frac{\|T^{(h)}\|_F^2}{|S_h|}. \label{eq:R0-bound-2} \numberthis
\end{align*}
Since $T^{(h)}$ involves only the rows of $V^1$ and $V^2$ indexed by $S_h$, the different $S_h$ are disjoint (forming a partition of $[N]$), and the rows of $V^1$ and $V^2$ are all i.i.d.\ by assumption, we see that the $T^{(h)}$ are independent random matrices.

As a heuristic aside, by the central limit theorem we expect that each $\|T^{(h)}\|_F^2 / |S_h|$ has roughly the distribution $\chi^2(\kappa^2)$.
Thus, the entire sum above should have roughly the distribution $\chi^2(\kappa^2 n)$.
If this were indeed exactly the case, then the tail bound we want would follow by (a simple variant of) Bernstein's inequality, the tails of this distribution looking, for large enough deviations, like those of $\chi^2(1)$, which precisely decay as $\exp(-t)$.

Now let us actually implement this reasoning.
By our previous observation, we may rewrite
\[ T^{(h)} = \sum_{i \in S_h} x_i^1 x_i^{2^{\top}}, \]
or, upon vectorizing this matrix,
\[ \vec(T^{(h)}) = \sum_{i \in S_h} x_i^1 \otimes x_i^{2}. \]
The summands here are i.i.d.\ random vectors.
Let us write $\nu^{(2)}$ for their law, which is the law of $x \otimes x^{\prime}$ with $x, x^{\prime} \sim \nu$ drawn independently.
Consider the moment generating function of such a random vector, $M: \R^{\kappa^2} \to \R$ given by
\[ M(\xi) \colonequals \Ex_{x, x^{\prime} \sim \nu} \exp(\langle \xi, x \otimes x^{\prime}). \]
Since $\nu$ has bounded support, this expectation is always finite.
By the generating function expansion of the moment generating function, the gradient and Hessian of $M$ at zero are
\begin{align*}
    \nabla M(0) &= \Ex_{x, x^{\prime} \sim \nu} x \otimes x^{\prime} \\ &= 0, \\
    \nabla^2 M(0) &= \Ex_{x, x^{\prime} \sim \nu} (x \otimes x^{\prime})(x \otimes x^{\prime})^{\top} \\
    &= \E (xx^{\top} \otimes x^{\prime}x^{\prime^{\top}}) \\
    &= (\Covx_{x \sim \nu}[x])^{\otimes 2} \\
    &\preceq I_{\kappa^2},
\end{align*}
the final observation following by our assumption on the covariance of $\nu$.
Taking a Taylor expansion, it then follows that, for all $\eta > 0$, there exists $\gamma > 0$ such that
\[ M(\xi) \leq \exp\left(\frac{1}{2}(1 + \eta) \|\xi\|^2\right) \text{ for all } \|\xi\| \leq \gamma. \]
Thus, for any choice of these parameters, each summand appearing in $\vec(T^{(h)})$ is $\gamma$-locally $(1 + \eta)$-subgaussian.
Accordingly, since these summands are independent, $\vec(T^{(h)})$ itself is $\gamma$-locally $(1 + \eta)|S_h|$-subgaussian, and $\vec(T^{(h)}) / \sqrt{|S_h|}$ is $\gamma\sqrt{|S_h|}$-locally $(1 + \eta)$-subgaussian.
Finally, define a vector $w \in \kappa^2 n$ to be the concatenation of the $\vec(T^{(h)}) / \sqrt{|S_h|}$ over $h = 1, \dots, n$.
Then, $w$ is $(\gamma \min_{h \in [n]} \sqrt{|S_h|})$-locally $(1 + \eta)$-subgaussian, and we have
\[ \|w\|^2 = \sum_{h = 1}^n \frac{\|T^{(h)}\|_F^2}{|S_h|}, \]
precisely the expression from our bound on $|R^{(0)}|$.
Since $|S_h| / N \to \rho_h > 0$ for all $h \in [n]$, for sufficiently large $N$ we will have $|S_h| \geq \frac{1}{2}\rho_{\min} N$ where $\rho_{\min} \colonequals \min_{h \in [n]} \rho_h$.
Thus, we find that, for $N$ sufficiently large, $w$ is $(\gamma\sqrt{\frac{1}{2}\rho_{\min}N})$-locally $(1 + \eta)$-subgaussian.

Now, we have
\[ |R^{(0)}| \leq (1 - \epsilon) \cdot \frac{1}{2} \|w\|^2. \]
By Proposition~\ref{prop:tail-bound}, taking $\delta = \eta$, we have that, for a constant $C^{\prime\prime} = C^{\prime\prime}(\eta)$, for all $0 \leq t \leq \gamma\, \frac{1 + \eta}{1 - \eta} \sqrt{\frac{1}{2}\rho_{\min}N} $
\[ \P[\|w\| \geq t] \leq C^{\prime\prime} \exp\left(-\frac{1 - \eta}{2(1 + \eta)}t^2\right) \]
Thus we find a tail bound on $R^{(0)}$:
\begin{align*}
    \P[|R^{(0)}| \geq t]
    &\leq \P\left[\|w\|^2 \geq \frac{2}{1 - \epsilon} t\right] \\
    &\leq C^{\prime\prime} \exp\left(-\frac{1 - \eta}{(1 + \eta)(1 - \epsilon)} t\right), \numberthis \label{eq:R0-tail} 
\end{align*}
which holds whenever $\sqrt{\frac{2}{1 - \epsilon} t} \leq \gamma \sqrt{\frac{1}{2}\rho_{\min} N}$, or equivalently whenever $t \leq \frac{1}{4}\gamma^2 (1 - \epsilon)\rho_{\min} N$.

Let us review the sequence of choices of parameters.
We determine $\epsilon \in (0, 1)$ according to the relation between $\beta^2$ and $\what{\mu}_1$, both of which are given in the setting of the Theorem.
We then choose $\eta$ depending on $\epsilon$ such that $\frac{1 - \eta}{(1 + \eta)(1 - \epsilon)} = 1 + \zeta > 1$ for some $\zeta > 0$.
The above reasoning then gives $\gamma$ depending on $\eta$ the parameter of local subgaussianity above.
Finally, the tail bound \eqref{eq:R0-tail} then holds for all $0 \leq t \leq \delta N$, for $\delta \colonequals \frac{1}{4}\gamma^2 (1 - \epsilon)\rho_{\min}$.

In summary, with these choices, we may take $A(N) \colonequals \delta N$ such that, for all $t \in [0, A(N)]$ we have
\[ \P[|R^{(0)}| \geq t] \leq C^{\prime\prime} \exp\left(-(1 + \zeta)t\right). \]
Thus, the conditions of Lemma~\ref{lem:overlap}, and finally we find that
\[ \limsup_{N \to \infty} \E \exp^{\leq D}(|R^{(0)}|) < \infty, \]
concluding the proof along with our previous reduction from an expectation over $R$ to the above one over $R^{(0)}$.
\end{proof}

\subsection{Statistical Lower Bound}

\begin{proof}
We follow the same outline as before; in this case, we want to show that
\[ \chi^2(\P_N \mid \Q_N) = \E \exp(R) - 1 \stackrel{(?)}{=} O(1). \]
Bounding $|R| \leq |R^{(0)}| + |R^{(1)}|$ as before and using that $|R^{(1)}| = O(1)$ almost surely under our assumptions, it suffices to show that $\E \exp(|R^{(0)}|) = O(1)$.
We reuse the bound from \eqref{eq:R0-bound-2} from the previous proof, which gives that, for some $\epsilon > 0$ depending on $\beta$,
\begin{align*} 
\E \exp(|R^{(0)}|)
&\leq \E \exp\left(\frac{1 - \epsilon}{2} \sum_{h = 1}^n \frac{\|T^{(h)}\|_F^2}{|S_h|}\right)
\intertext{and, using the independence of the matrices $T^{(h)}$, we may factorize this as}
&= \prod_{h = 1}^n \E \exp\left(\frac{1 - \epsilon}{2} \cdot \frac{\|T^{(h)}\|_F^2}{|S_h|}\right)
\intertext{Recall that for this result we have the extra assumption that, when $x, x^{\prime} \sim \nu$ are independent, then $x \otimes x^{\prime}$ is a 1-subgaussian random vector. Further, as we used above, the vectorization $\vec(T^{(h)})$ is a sum of $|S_h|$ independent vectors having this same distribution, so in particular $\vec(T^{(h)}) / \sqrt{|S_h|}$ is again a 1-subgaussian random vector. By Proposition~\ref{prop:moment-bound}, all of the expectations in this product are finite, and thus since there is a fixed number $n$ of terms in the product, we have}
&= O(1),
\end{align*}
completing the proof.
\end{proof}

\subsection{Growth of Low-Degree \texorpdfstring{$\chi^2$-}{Chi-Squared }Divergence: Proof of Theorem~\ref{thm:block-low-deg-div}}
\label{sec:pf:thm:block-low-deg-div}

We adapt some of the calculations from Section~\ref{section:block-constant-comp}.
Recall that the low-degree $\chi^2$-divergence is
\begin{align*}
    1 + \chi^2_{\leq D}(\P_N \mid \Q_N)
    &= \E \exp^{\leq D}(R) \\
    &= \E \exp^{\leq D}(R^{(0)} + R^{(1)}),
\end{align*}
where $|R^{(1)}|$ is bounded by an absolute constant.
Thus, by Proposition~\ref{prop:exp-bounds}, it suffices to show that
\[ \E \exp^{\leq D}(R^{(0)}) = \omega(1) \]
as $N \to \infty$.

Recall also that, setting $u_{ab} \colonequals P(v_a^1 \odot v_b^2)$, we have
\[ R^{(0)} = \frac{\beta^2}{2N} \sum_{a, b = 1}^{\kappa} u_{ab}^{\top} \Gamma u_{ab} = \left\langle \Gamma, \sum_{a, b = 1}^{\kappa} u_{ab}u_{ab}^{\top}\right\rangle. \]
Let us define
\[ M \colonequals \sum_{a, b = 1}^{\kappa} u_{ab}u_{ab}^{\top} \in \R^{n \times n}_{\sym}. \]
The entries of this matrix are
\begin{align*} 
M_{rs} 
&= \sum_{a, b = 1}^{\kappa} (u_{ab})_r (u_{ab})_s \\
&= \sum_{a, b = 1}^{\kappa} \left(\frac{1}{\sqrt{|S_r|}} \sum_{i \in S_r} (v_a^1)_i(v_b^2)_i\right)\left(\frac{1}{\sqrt{|S_s|}} \sum_{i \in S_s} (v_a^1)_i(v_b^2)_i\right) \\
&= \sum_{a, b = 1}^{\kappa} \left(\frac{1}{\sqrt{|S_r|}} \sum_{i \in S_r} (x_i^1 \otimes x_i^2)_a \right)\left(\frac{1}{\sqrt{|S_s|}} \sum_{i \in S_s} (x_i^1 \otimes x_i^2)_b \right) \\
&= \left\langle \frac{1}{\sqrt{|S_r|}} \sum_{i \in S_r} x_i^1 \otimes x_i^2, \frac{1}{\sqrt{|S_s|}} \sum_{i \in S_s} x_i^1 \otimes x_i^2 \right\rangle.
\end{align*}

Recall that here the $x_i^1$ and $x_i^2$ for $i \in [N]$ are all i.i.d.\ draws from the distribution $\nu$ on $\R^{\kappa}$, and the $S_r$ are a partition of $[N]$ each of whose blocks has diverging size.
Let us write $\Sigma \colonequals \Cov_{x \sim \nu}[x]$.
Since we also assume that $\nu$ is centered, we have $\Cov_{x^1, x^2 \sim \nu}[x^1 \otimes x^2] = \Sigma^{\otimes 2}$.
By the central limit theorem, we then have the convergence in distribution
\[ \left(\frac{1}{\sqrt{|S_1|}} \sum_{i \in S_1} x_i^1 \otimes x_i^2, \dots, \frac{1}{\sqrt{|S_n|}} \sum_{i \in S_n} x_i^1 \otimes x_i^2\right) \dto (g_1, \dots, g_n) \]
where $g_i \sim \sN(0, \Sigma \otimes \Sigma)$ are i.i.d.\ Gaussian random vectors in $\R^{\kappa^2}$.
Writing $G \in \R^{\kappa^2 \times n}$ for the matrix with these $g_i$ as its columns, we then also have
\[ M \dto G^{\top} G, \]
and therefore, using also that $\frac{1}{N} \Gamma \to \widetilde{\Gamma}$ entrywise,
\[ R^{(0)} = \frac{\beta^2}{2N} \langle \Gamma, M \rangle \dto \frac{\beta^2}{2}\langle \widetilde{\Gamma}, G^{\top}G \rangle = \frac{\beta^2}{2}\Tr(G \widetilde{\Gamma} G^{\top}). \]
Since the underlying distribution $\nu$ is bounded, a routine approximation argument shows that the moments of $R^{(0)}$ converge to the corresponding moments of the right-hand side as well: for any fixed $d \geq 0$,
\begin{align*}
\E R^{(0)^d} 
&\to \left(\frac{\beta^2}{2}\right)^d \E\bigg(\Tr(G \widetilde{\Gamma} G^{\top})\bigg)^d \intertext{Next, we ``whiten'' the distribution of $G$: introducing $H$ of the same shape with i.i.d.\ entries distributed as $\sN(0, 1)$, $G$ has the same law as $(\Sigma^{\otimes 2})^{1/2}H$. In terms of $H$, we may rewrite}
&= \left(\frac{\beta^2}{2}\right)^d \E\bigg(\Tr((\Sigma^{\otimes 2})^{1/2} H \widetilde{\Gamma} H^{\top}(\Sigma^{\otimes 2})^{1/2})\bigg)^d \\
&= \left(\frac{\beta^2}{2}\right)^d \E\bigg(\Tr(\Sigma^{\otimes 2} H \widetilde{\Gamma} H^{\top})\bigg)^d
\intertext{Further, by the rotational invariance of $H$, we may diagonalize the two deterministic matrices appearing here without affecting the value. Letting $\lambda(X)$ denote the vector of eigenvalues of a symmetric $X$, we have}
&= \left(\frac{\beta^2}{2}\right)^d \E\bigg(\Tr(\Diag(\lambda(\Sigma^{\otimes 2})) H \Diag(\lambda(\widetilde{\Gamma})) H^{\top})\bigg)^d \\
&= \left(\frac{\beta^2}{2}\right)^d \E\left(\sum_{a, b = 1}^{\kappa} \sum_{r = 1}^n \lambda_a(\Sigma) \lambda_b(\Sigma) \lambda_r(\widetilde{\Gamma}) H_{(a, b), r}^2 \right)^d,
\intertext{where we identify $[\kappa^2]$ with $[\kappa] \times [\kappa]$ in indexing the rows of $H$. Finally, Proposition~\ref{prop:baopaper} gives a lower bound on such moments of quadratic forms,}
&\geq \left(\frac{\beta^2}{2}\right)^d \cdot 2^{d - 1}(d - 1)! \sum_{a, b = 1}^{\kappa} \sum_{r = 1}^n \lambda_a(\Sigma)^d \lambda_b(\Sigma)^d \lambda_r(\widetilde{\Gamma})^d
\intertext{Here, we have $\|\Sigma\| = 1$ and $\Sigma \succeq 0$, so $\lambda_1(\Sigma) = 1$, while $\lambda_1(\widetilde{\Gamma}) = \what{\mu}_1$ and so, for $d$ even,}
&\geq \frac{(d - 1)!}{2} (\beta^2 \what{\mu}_1)^d
\end{align*}

Now, we have
\begin{align*}
    \E \exp^{\leq D}(R^{(0)}) 
    &= \sum_{d = 0}^D \frac{1}{d!} \E R^{(0)^d}
    \intertext{Here, since we can write $R^{(0)} = \langle U^1, U^2 \rangle$ for two suitable i.i.d.\ random vectors $U^i$, every term is non-negative. So, we may bound below by only the even terms, and fixing some $2d_0$ not depending on $N$, for sufficiently large $N$ we have $D = D(N) \geq 2d_0$ as $D(N) \to \infty$ by assumption. So, we may further bound below by only the even $d$ terms between $d = 2$ and $d = 2d_0$, obtaining after reindexing}
    &\geq \sum_{d = 1}^{d_0} \frac{1}{(2d)!} \E R^{(0)^{2d}} \\
    &\geq \sum_{d = 1}^{d_0} \frac{1}{(2d)!} \cdot \frac{(2d - 1)!}{2} (\beta^2 \what{\mu}_1)^{2d} \\
    &= \frac{1}{4}\sum_{d = 1}^{d_0} \frac{1}{d} (\beta^2 \what{\mu}_1)^{2d}.
\end{align*}
Since by assumption $\beta^2 \what{\mu}_1 > 1$, we obtain a diverging lower bound taking $d_0 \to \infty$.

\section{General Variance Profiles: Proof of Theorem~\ref{theorem:general-bound}}
\label{sec:pf:theorem:general-bound}

In this section, we prove our main result on the case of general variance profiles $\Delta$, not necessarily having block structure.
Our approach in this case will be quite different.
In the block-structured case, we have seen that the low-degree $\chi^2$-divergence calculations become essentially finite-dimensional after appropriately reorganizing to take advantage of the averaging occurring over blocks.
Here, recall that we assume that the signal has rank 1, and thus our overlap random variable $R$ takes the form
\[ R = \frac{\beta^2}{2N} (x^1 \odot x^2)^{\top} \Phi^{\prime} (x^1 \odot x^2), \]
where $\Phi^{\prime} = \Phi + \Diag(\Phi)$ is our slight modification of the entrywise reciprocal of the variance profile matrix (as we will see, just like in the block-structured case, the second term here is inconsequential).

Recall that our initial calculations gave the formula
\begin{align*}
\chi^2_{\leq D}(\P \mid \Q) 
&= \E\left[\exp^{\leq D}(R)\right] - 1
\intertext{In the coming proof, we will expand this into a power series,}
&= \sum_{d = 1}^D \frac{1}{d!} \E R^d,
\end{align*}
and will further expand $\E R^d$ and evaluate or bound the expectations over $x^1$ and $x^2$ that appear.
In each such term, we will then be left with a complicated polynomial of $\Phi$.
So, before continuing to the proof itself, we develop some tools for bounding the polynomials that will appear.

\subsection{Tools for Graph Sums of a Matrix}
\label{sec:graph-sums}

\subsubsection{Algebraic Properties}

The specific polynomials we will need to control are as follows.
\begin{definition}[Scalar graph functions of a matrix]
\label{definition:graph-function}
    Let $G = (V, E)$ be a multigraph (that is, allowing self-loops and parallel edges) with labelled vertex set $V = [M]$, and let $\Phi \in \R^{N \times N}_{\sym}$.
    \begin{align*}
        f_{G}(\Phi) &\coloneqq \sum_{i: [M] \hookrightarrow [N]} \prod_{\{x, y\} \in E(G)} \Phi_{i(x)i(y)} \\
        \tilde{f}_{G}(\Phi) &\coloneqq \sum_{i: [M] \to [N]} \prod_{\{x, y\} \in E(G)} \Phi_{i(x)i(y)}
    \end{align*}
    The difference between the two expressions is that, in the first one, the ``labelling function'' $i$ is constrained to be injective, while in the second it is not.
\end{definition}

The $f_G(\Phi)$ are scalar-valued versions of the \emph{graph matrices} that have found applications in the analysis of sum-of-squares optimization and, more recently, other algorithmic settings \cite{AMP-2016-GraphMatrices,BHKKMP-2019-PlantedClique,HKPRSS-2017-SOSSpectral,HK-2022-RefutingRandomPolynomialSystems,VTW-2022-NearOptimalFittingEllipsoids,HKPX-2023-EllipsoidFittingConstant}.
The $\tilde{f}_G(\Phi)$ have appeared in the theories of free probability and traffic probability as well \cite{BS-2010-SpectralAnalysisRandomMatrices,MS-2012-SharpBoundsGraphMatrices,ACDGM-2021-FreenessOverDiagonal}.
Both can be used as bases for the polynomials of (the entries of) $\Phi$ that are invariant under the two-sided action of the symmetric group; for further discussion on this perspective see \cite{KMW-2024-TensorCumulantsInvariantInference}.

For many of the results we will need, it will turn out to be more convenient to form graphs by ``gluing'' certain vertices of a perfect matching.
We define notation to work with this notion as follows.
The below will rely on our definitions concerning partitions from Section~\ref{sec:combinatorics}.
\begin{definition}[Partition representing a graph]
\label{def:partition-graph}
Let us identify $[2d]$ with the more convenient indexing set $\mathcal{I} = \{1,1',2,2',\ldots,d,d'\}$.
Let $\pi$ be a partition of $\sI$, and let $b: \sI \to \pi$ be the function so that $b(x)$ equals the block of $\pi$ that $x$ belongs to.
We write $G(\pi) = (V, E)$ for the multigraph with $|\pi|$ vertices, identified with $\pi$ itself, whose multiset of edges is $E = \{\{b_{\pi}(x), b_{\pi}(x')\}: x \in [d]\}$.
\end{definition}
\noindent
We use the same notation to refer to sums associated to graphs associated to a partition:
\begin{definition}
\label{definition:partition-function}
    For $\pi$ a partition of $\sI$ as above, we write
    \begin{align*}
        f_{\pi}(\Phi) &\colonequals f_{G(\pi)}(\Phi) = \sum_{i: \pi \hookrightarrow [N]} \prod_{x = 1}^d \Phi_{i(b(x))i(b(x'))}, \\
        \tilde{f}_{\pi}(\Phi) &\colonequals \tilde{f}_{G(\pi)}(\Phi) = \sum_{i: \pi \to [N]} \prod_{x = 1}^d \Phi_{i(b(x))i(b(x'))}.
    \end{align*}
\end{definition}

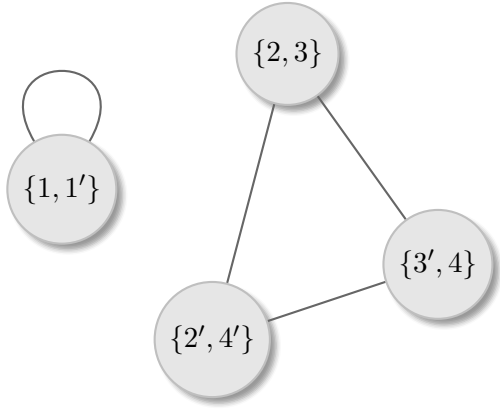
\begin{figure}
\begin{center}
\begin{tikzpicture}[
    node/.style={
        circle, draw=gray!50, fill=gray!20, thick,
        minimum size=1.2cm, font=\bfseries, blur shadow
    },
    edge/.style={
        draw=black!60, thick
    },
    loop edge/.style={
        edge, loop above, looseness=5, in=120, out=60
    }
]

\node[node] (1) at (0,0) {$\{1,1'\}$};
\node[node] (2) at (3,1.8) {$\{2,3\}$};
\node[node] (3) at (5,-1) {$\{3', 4\}$};
\node[node] (4) at (2,-2) {$\{2', 4'\}$};

\draw[loop edge] (1) to (1);
\draw[edge] (2) -- (3);
\draw[edge] (2) -- (4);
\draw[edge] (3) -- (4);

\end{tikzpicture}
\vspace{0.5em}
\caption{The graph $G(\pi)$ represented by the partition $\pi =\{\{1,1'\}, \{2,3\}, \{3',4\}, \{4',2'\}\}$ of the set $\sI = \{1, 1', 2, 2', 3, 3', 4, 4'\}$.} 
\label{figure1}
\end{center}
\end{figure}

\begin{example}
To illustrate, for the case $d=4$ and $\pi = \{(1,1'), (2,3), (3',4), (2',4')\}$, the corresponding graph $G(\pi)$ is represented by Figure~\ref{figure1}.
The associated polynomials of a matrix are:
\begin{align*}
f_\pi(\Phi) &= f_{G(\pi)}(\Phi) = \sum_{\substack{i,j,k,l \text{ distinct}}} \Phi_{ii} \Phi_{jl} \Phi_{jk} \Phi_{kl}, \\
    \tilde{f}_\pi(\Phi) &= \tilde{f}_{G(\pi)}(\Phi) =  \sum_{i,j,k,l} \Phi_{ii} \Phi_{jl} \Phi_{jk} \Phi_{kl} = \Tr(\Phi)\Tr(\Phi^3).
\end{align*}
\end{example}

The following describes the simple but important relationship between the two polynomial families $f_G$ and $\tilde{f}_G$.
\begin{proposition}
\label{prop:partition-rearrangement}
    Let $\pi$ be a partition of $\sI$ as above.
    Then, for any $\Phi \in \R^{N \times N}_{\sym}$,
    \[ \tilde{f}_{\pi}(\Phi) = \sum_{\rho \preceq \pi} f_{\pi}(\Phi). \]
\end{proposition}
\begin{proof}
    Given $\tilde{i}: \pi \to [N]$, there are a unique $\rho \preceq \pi$ and $i: \rho \hookrightarrow [N]$ such that, for $q: \pi \to \rho$ as above, $\tilde{i}(S) = i(q(S))$ for each $S \in \pi$.
    In words, this $\rho$ is formed by merging the blocks of $\pi$ on which $\tilde{i}$ is constant, and $i$ is the same function of $\tilde{i}$ but on these coarser blocks.
    Now, working from the definition,
    \begin{align*}
        \tilde{f}_{\pi}(\Phi) 
        &= \sum_{\tilde{i}: \pi \to [N]} \prod_{x = 1}^d \Phi_{\tilde{i}(b_{\pi}(x))\tilde{i}(b_{\pi}(x'))} \\
        &= \sum_{\rho \preceq \pi} \sum_{i: \rho \hookrightarrow [N]} \prod_{x = 1}^d \Phi_{i(b_{\rho}(x))i(b_{\rho}(x'))} \\
        &= \sum_{\rho \preceq \pi} f_{\rho}(\Phi),
    \end{align*}
    completing the proof.
\end{proof}
\noindent
One may also invert the above relation using M\"{o}bius inversion in the partially ordered set of partitions, but we will only need the simpler above direction.

\subsubsection{Matchings and 2-Regular Graphs}

The case of perfect matchings $\pi \in \Match(\sI)$ in the above setting will play a crucial role throughout.
There are two coincident reasons for the importance of the perfect matchings: on the one hand, they are the even partitions with the greatest number of parts, and thus ``entropically dominate'' certain calculations we will encounter; on the other hand, the only cases of $f_{\pi}(\Phi)$ or $\tilde{f}_{\pi}(\Phi)$ that have a natural interpretation in terms of the spectrum of $\Phi$ are the $\tilde{f}_{\pi}(\Phi)$ where $\pi \in \Match(\sI)$, as appears in the above example.

Indeed, for $\pi$ a matching we have that $G(\pi)$ is a 2-regular multigraph on $d$ vertices, or a disjoint union of several cycles (this is a multigraph because some of the cycles may be self-loops or pairs of parallel edges forming 2-cycles).
\begin{definition}
    For $G$ a 2-regular graph, we write $\cyc(G)$ for the set of cycles of $G$.
    Formally, this may be viewed as $\cyc(G) \in \Part(V(G))$, a partition of the vertices into those belonging to each cycle of $G$.
\end{definition}

It is well-known that one may sample uniformly from the set of 2-regular multigraphs on $d$ vertices using the \emph{configuration model}: to each vertex we assign two ``half-edges'' or ``stubs'' (for a total of $2d$ half-edges), and then form a graph by linking half-edges along a uniformly random perfect matching.
Analyzing this probability measure, the following result has been derived in prior work.

\begin{proposition}[\cite{tishby2023distribution}]
\label{prop:graph-moments}
Let $\mathcal{U}$ be the set of all $2$-regular multigraphs on the labelled vertex set $[d]$ (which also must have $d$ edges). Let $\mathcal{G}$ be the uniform distribution over $\mathcal{U}$. Then, for any $a > 0$,
\begin{align*}
    |\mathcal{U}| &= (2d - 1)!!, \\
    \Ex_{G \sim \mathcal{G}}\left[a^{|\cyc(G)|}\right] &= \frac{2^d}{(2d-1)!!} \frac{\Gamma\left(d + \frac{a}2\right)}{\Gamma\left(\frac{a}2\right)}.
\end{align*}
Equivalently, 
$$\sum_{G \in \mathcal{U}} a^{|\cyc(G)|} = 2^d \cdot \frac{\Gamma\left(d + \frac{a}2\right)}{\Gamma\left(\frac{a}2\right)} = \prod_{j = 1}^d (a + 2j - 2).$$
\end{proposition}

One may also take a ``dual'' viewpoint to the configuration model: suppose that we have a fixed perfect matching on a set of vertices of size $2d$ (such as $\sI$ discussed above).
Form a uniformly random perfect matching of those vertices, where the $d$ pairs of this perfect matching are also uniformly labelled by $[d]$.
Identifying each pair of vertices joined in this perfect matching and giving the resulting vertex the label of that pair yields a random graph in $\mathcal{U}$.
By symmetry considerations we see that this is again a uniformly random graph in $\mathcal{U}$, i.e., that this has the law $\mathcal{G}$.
But, what we have described above is precisely the process of drawing $\pi \sim \Unif(\Match(\sI))$ and forming $G(\pi)$, except that we endow the pairs of $\pi$ with a random labelling.
Since this labelling has no effect on $|\cyc(G(\pi))|$, we find the following corollary that will be useful in our calculations.

\begin{corollary}
    \label{cor:moment}
    For any $a > 0$,
    \[ \sum_{\pi \in \Match(\sI)} a^{|\cyc(G(\pi))|} = 2^d \cdot \frac{\Gamma\left(d + \frac{a}2\right)}{\Gamma\left(\frac{a}2\right)} = \prod_{j = 1}^d (a + 2j - 2). \]
\end{corollary}

\subsubsection{Sharpened Bounds on Positive Graph Sums}

We will be interested in bounding the $f_G(\Phi)$ and $\tilde{f}_G(\Phi)$ expressions for $\Phi$ the inverse variance profile matrix discussed previously.
Recall that this matrix has non-negative entries, and in the block-structured case is low-rank.
Our assumptions also roughly impose that all entries of $\Phi$ are of order $\Theta(1)$ (we actually only require a uniform upper bound, but we think of all entries being of the same constant order as the ``typical case,'' which for example is true in the case of block-structured models).

While these polynomials have appeared in the literature before and some bounds have been proved on them, usually these bounds are more effective when neither of the above assumptions holds.
The following is the one of the main and most useful bounds on these expressions stemming from the free probability literature mentioned above; for the sake of exposition we restrict our attention to a special class of graphs.
\begin{proposition}[\cite{BS-2010-SpectralAnalysisRandomMatrices,MS-2012-SharpBoundsGraphMatrices}]
    Suppose that $G$ is a 2-edge-connected graph.
    Then,
    \[ \tilde{f}_G(\Phi) \leq N \cdot \|\Phi\|^{|E(G)|}. \]
\end{proposition}

The following example shows that, in the simplest choice of $\Phi$ that will arise in our setting, this bound is already quite loose.
\begin{example}[All-ones matrix]
    If $\Phi = 11^{\top}$, then we have $\|\Phi\| = N$, so the right-hand side of the bound above is $N^{|E(G)| + 1}$.
    On the other hand, the actual value is easily computed to be $\tilde{f}_G(\Phi) = N^{|V(G)|}$.
    Since for a connected graph $G$ we have $|V(G)| = |E(G)| + 1$ if and only if $G$ is a tree, which is never 2-edge-connected, we see that on 2-edge-connected graphs the above bound is always suboptimal on this $\Phi$.
\end{example}
\noindent
Our main goal here will be to develop bounds that behave, for the $\Phi$ with the rough structure outlined above, more like $\tilde{f}_G(\Phi) \lesssim N^{|V(G)|}$.

We build up some tools for increasingly more general bounds on the $\tilde{f}_G(\Phi)$, starting from simple $G$ and using those bounds to treat more general $G$.

\begin{proposition}[Paths]
    Let $G$ be a path on $k \geq 1$ vertices (where the path on $k = 1$ vertex is a single isolated vertex).
    Then,
    \[ \tilde{f}_G(\Phi) \leq N \|\Phi\|^{k - 1} = \left(\frac{N}{\|\Phi\|}\right) \|\Phi\|^k. \]
\end{proposition}

\begin{proposition}[Cycles]
\label{prop:cycles}
    Let $G$ be a cycle on $k \geq 1$ vertices (where the cycle on $k = 1$ vertex is a single self-loop).
    Then, both of the following hold:
    \begin{align*}
        \tilde{f}_G(\Phi) &\leq \rank(\Phi) \|\Phi\|^k, \\
        \tilde{f}_G(\Phi) &\leq N \|\Phi\|_{\ell^{\infty}} \|\Phi\|^{k - 1} \\
        &= \left(\frac{N \|\Phi\|_{\ell^{\infty}}}{\|\Phi\|}\right) \|\Phi\|^k.
    \end{align*}
\end{proposition}

To compare the above bounds, recall that we are interested in $\Phi$ that behave like the all-ones matrix.
In particular, we expect to have $\|\Phi\|_{\ell^{\infty}} = O(1)$, while $\|\Phi\| = \Theta(N)$.
Thus when $\rank(\Phi) = O(1)$ then the above two bounds are of the same order, while when $\rank(\Phi) = \omega(1)$ the second bound can be smaller than the first.
For this reason, and to avoid our results depending on the more fragile property of the rank, we will rely on the second bound above.

In general, one may of course also bound $\tilde{f}_{G}(\Phi) \leq N^k \|\Phi\|_{\ell^{\infty}}^k$ in the above setting, but we will see that to obtain sharp results it is important for us that it is $\|\Phi\|^k$ that appears in our bound, not $(N \cdot \|\Phi\|_{\ell^{\infty}})^k$, which may be viewed as replacing $\|\Phi\|$ by a coarse entrywise bound on this norm.
The situation in our application will be delicate: it will be acceptable for this coarse bound expression to appear linearly, as it does above, but not with the possibly much larger exponent $k$.

Finally, we show how to control much more general graph sums by reducing them to the above case.

\begin{proposition}
    Suppose that $G$ and $H$ are graphs on the same vertex set, and $H$ is formed by deleting $t$ edges from $G$.
    Then,
    \[ \tilde{f}_G(\Phi) \leq \|\Phi\|_{\ell^{\infty}}^t \tilde{f}_H(\Phi). \]
\end{proposition}

\begin{proposition}
    Let $G$ be a multigraph with $|V(G)| = m$ and $|E(G)| = d$ where every vertex has degree at least 2.
    Then, it is possible to remove at most $2d - 2m \geq 0$ edges from $G$ to leave a graph consisting of a disjoint union of paths (including isolated vertices) and cycles (including isolated self-loops), and having at most $2d - 2m + |\conn(G)|$ such connected components.
\end{proposition}
\begin{proof}
    For each vertex $v$ of $G$, choose an arbitrary $\deg(v) - 2 \geq 0$ edges adjacent to that vertex to remove, allowing edges to be repeated in these choices.
    The total number of edges removed is at most $\sum_{v \in V(G)} (\deg(v) - 2) = 2d - 2m$.
    In the resulting graph, every vertex has degree at most 2, and thus all connected components are either paths or cycles.
    Further, removing a single edge from a graph can only increase the number of connected components by at most 1.
    Since initially $G$ has $|\conn(G)|$ connected components, the result on the number of connected components follows as well.
\end{proof}

\begin{corollary}
    \label{cor:graph-sum-bound}
    Let $G$ be a multigraph with $|V(G)| = m$ and $|E(G)| = d$ where every vertex has degree at least 2.
    Then,
    \[ \tilde{f}_G(\Phi) \leq \left(\frac{N(1 + \|\Phi\|_{\ell^{\infty}})^2}{\|\Phi\|}\right)^{2d - 2m + |\conn(G)|} \|\Phi\|^{m}. \]
\end{corollary}
\begin{proof}
    Let $H$ be the graph formed from $G$ by deleting at most $2d - 2m$ edges to leave a disjoint union of paths and cycles.
    For any connected component $C$ of $H$, by the above results we have
    \[ \tilde{f}_C(\Phi) \leq \left(\frac{N(1 + \|\Phi\|_{\ell^{\infty}})}{\|\Phi\|}\right) \|\Phi\|^{|V(C)|}. \]
    Combining the previous results and using the multiplicativity of $\tilde{f}_G$ over disjoint unions of graphs, we have
    \begin{align*}
        \tilde{f}_G(\Phi)
        &\leq (1 + \|\Phi\|_{\ell^{\infty}})^{2d - 2m} \tilde{f}_H(\Phi) \\
        &\leq (1 + \|\Phi\|_{\ell^{\infty}})^{2d - 2m} \prod_{C \in \conn(H)} \tilde{f}_C(\Phi) \\
        &\leq (1 + \|\Phi\|_{\ell^{\infty}})^{2d - 2m} \prod_{C \in \conn(H)} \left(\frac{N(1 + \|\Phi\|_{\ell^{\infty}})}{\|\Phi\|}\right) \|\Phi\|^{|V(C)|} \\
        &= (1 + \|\Phi\|_{\ell^{\infty}})^{2d - 2m} \left(\frac{N(1 + \|\Phi\|_{\ell^{\infty}})}{\|\Phi\|}\right)^{|\conn(H)|} \|\Phi\|^{m} \\
        &\leq (1 + \|\Phi\|_{\ell^{\infty}})^{2d - 2m} \left(\frac{N(1 + \|\Phi\|_{\ell^{\infty}})}{\|\Phi\|}\right)^{2d - 2m + |\conn(G)|} \|\Phi\|^{m} \\
        &\leq \left(\frac{N(1 + \|\Phi\|_{\ell^{\infty}})^2}{\|\Phi\|}\right)^{2d - 2m + |\conn(G)|} \|\Phi\|^{m},
    \end{align*}
    as claimed.
\end{proof}

Specifically, the following is the way in which we will apply the above ideas.
Recall that we have mentioned that our arguments will involve showing that some sums over general partitions are dominated by the contribution of the matchings.
To implement this proof idea, it will be useful to \emph{compare} the $\tilde{f}_{G(\rho)}$ for general partitions $\rho$ to the same terms for matchings.
The following precisely gives a way to do this.
\begin{lemma}
\label{lemma:term-bound}
    Suppose that $\rho \preceq \pi$ is such that $\pi \in \Match(\sI)$. Then,
    \[ \tilde{f}_{G(\rho)}(\Phi) \leq \left(\frac{N(1 + \|\Phi\|_{\ell^{\infty}})^2}{\|\Phi\|}\right)^{2d-2m + |\cyc(G(\pi))|} \|\Phi\|^{m} \]
\end{lemma}
\begin{proof}
    The result is immediate from the previous one once we note that, since $G(\rho)$ is formed by identifying certain vertices of $G(\pi)$, we have $|\conn(G(\rho))| \leq |\conn(G(\pi))| = |\cyc(G(\pi))|$.
\end{proof}

\subsection{Computational Lower Bound}

Let us recall the assumptions of the theorem: we have
\begin{align*}
    D = D(N) &\ll \frac{\lambda_1(\Phi)^3}{N^2}, \\
    \|\Phi\|_{\ell^{\infty}} &\leq B,
\end{align*}
where $B$ is a constant independent of $N$.
Recall from our preliminary calculations in Section~\ref{sec:additive-gaussian} that we will in effect need to replace $\Phi$ by $\Phi^{\prime} \colonequals \Phi + \Diag(\Phi)$.
However, as the following relations express, both the entrywise $\ell^{\infty}$ and operator norms are not changed much by this modification:
\begin{align*}
    \lambda_1(\Phi) \leq \lambda_1(\Phi^{\prime}) &\leq \lambda_1(\Phi) + B, \\
    B \leq \|\Phi^{\prime}\|_{\ell^{\infty}} &\leq 2B.
\end{align*}

We take as a starting point our formula for the low-degree $\chi^2$-divergence in a general inhomogeneous spiked Wigner matrix model from Lemma~\ref{lem:low-deg-prelim}.
Let us denote
\[ u \colonequals x^1 \odot x^2, \]
where we recall that $x^1$ and $x^2$ have i.i.d.\ entries drawn from $\nu$.
Therefore, $u$ has i.i.d.\ entries drawn from $\nu^{(2)}$, the law of $xx^{\prime}$ for $x, x^{\prime} \sim \nu$.
Then, the Lemma says that we have
\[ \chi^2_{\leq D}(\P_N \mid \Q_N) = \sum_{d = 1}^D \frac{1}{d!} \E R^d = \sum_{d = 1}^D \frac{1}{d!} \left(\frac{\beta^2}{2N}\right)^d \E(u^{\top}\Phi^{\prime} u)^d. \]

Expanding the remaining expectations, we have
\begin{align*}
   \E(u^{\top}\Phi^{\prime} u)^d = \sum_{i_1, \dots, i_d, j_1, \dots, j_d \in [N]} \left(\prod_{k = 1}^d \Phi^{\prime}_{i_kj_k}\right) \cdot \E\left[\prod_{k = 1}^d u_{i_k}u_{j_k}\right]
\end{align*}
Rephrasing, let us introduce the indexing set
\[ \sI \colonequals \{1, 1', 2, 2', \dots, d, d'\}. \]
Then, we may write
\begin{align*}
   \E(u^{\top}\Phi^{\prime} u)^d 
   &= \sum_{i: \sI \to [N]} \left(\prod_{k = 1}^d \Phi^{\prime}_{i(k)i(k')}\right) \cdot \E\left[\prod_{k = 1}^d u_{i(k)}u_{i(k')}\right]
   \intertext{and, since $\nu^{(2)}$ is a symmetric distribution, any term where any element of $[N]$ occurs an odd number of times in the image of $i$ is zero and we may restrict}
   &= \sum_{\substack{i: \sI \to [N] \\ |i^{-1}(j)| \text{ is even for all } j \in [N]} } \left(\prod_{k = 1}^d \Phi^{\prime}_{i(k)i(k')}\right) \cdot \E\left[\prod_{k = 1}^d u_{i(k)}u_{i(k')}\right]
   \intertext{Equivalently, we may associate to each $i$ a partition $\rho$ of $\sI$ into even parts, where the elements of $\sI$ that $i$ maps to the same element of $[N]$ belong to the same block of $\rho$. In these terms, we can rewrite}
   &= \sum_{\rho \in \EvenPart(\sI)} \, \sum_{i: \rho \hookrightarrow [N]} \left(\prod_{k = 1}^d \Phi^{\prime}_{i(k)i(k')}\right) \cdot \prod_{A \in \pi} \E\left[u_1^{|A|}\right]
   \intertext{where we have used that the entries of $u$ are i.i.d. 
   Recall that $u_1$ has the law of $xx^{\prime}$ for $x, x^{\prime} \sim \nu$, and so by assumption $u_1$ is $\sigma^2$-subgaussian. In particular, by Proposition~\ref{prop:subgauss-tau}, there exists $\tau > 0$ such that its even moments are bounded by $\E u_1^{2k} \leq \tau^{2k - 2} (2k - 1)!!$. We may then expand and bound, noting that all terms appearing are non-negative,}
   &= \sum_{\rho \in \EvenPart(\sI)} \, \sum_{i: \rho \hookrightarrow [N]} \left(\prod_{k = 1}^d \Phi^{\prime}_{i(k)i(k')}\right) \cdot \prod_{A \in \rho} \tau^{2|A| - 2} (|A| - 1)!! \\
   &= \sum_{\rho \in \EvenPart(\sI)} \tau^{2d - 2|\rho|}  \prod_{A \in \rho} (|A| - 1)!! \sum_{i: \rho \hookrightarrow [N]} \prod_{k = 1}^d \Phi^{\prime}_{i(k)i(k')} \\
   &= \sum_{\rho \in \EvenPart(\sI)} \tau^{2d - 2|\rho|} \cdot \prod_{A \in \rho} (|A| - 1)!! \cdot f_{\rho}(\Phi^{\prime}) \numberthis \label{eq:general-comp-E-quad-form}
   \intertext{Now, note that the product of $(|A| - 1)!!$ precisely counts the number of matchings that are refinements of $\rho$. Therefore, we may reorganize}
   &= \sum_{\rho \in \EvenPart(\sI)} \sum_{\substack{\pi \in \Match(\sI) \\ \rho \preceq \pi}} \tau^{2d - 2|\rho|} f_{\rho}(\Phi^{\prime}) \\
   &= \sum_{\pi \in \Match(\sI)} \sum_{\substack{\rho \in \EvenPart(\sI) \\ \rho \preceq \pi}} \tau^{2d - 2|\rho|} f_{\rho}(\Phi^{\prime})
   \intertext{Here, our idea will be to show that the dominant contribution in each inner sum comes from $\pi$ itself. It turns out that we do not need to treat $\pi$ specially to see this; this fact is already contained in the estimate of Corollary~\ref{cor:graph-sum-bound}. We first rewrite}
   &= \sum_{\pi \in \Match(\sI)} \sum_{m = 1}^d \tau^{2d - 2m} \sum_{\substack{\rho \in \EvenPart_m(\sI) \\ \rho \preceq \pi}} f_{\rho}(\Phi^{\prime})
   \intertext{Where we may now bound by first bounding all $f$ by corresponding $\tilde{f}$ and then using Corollary~\ref{cor:graph-sum-bound}, setting $K \colonequals \frac{N(1 + \|\Phi^{\prime}\|_{\ell^{\infty}})^2}{\|\Phi^{\prime}\|}$ and $L \colonequals \tau^2(1 + \|\Phi^{\prime}\|_{\ell^{\infty}})^4$,}
   &\leq \sum_{\pi \in \Match(\sI)} \sum_{m = 1}^{d} \tau^{2d - 2m} \sum_{\substack{\rho \in \EvenPart_m(\sI) \\ \rho \preceq \pi}} \left(\frac{N(1 + \|\Phi^{\prime}\|_{\ell^{\infty}})^2}{\|\Phi^{\prime}\|}\right)^{2d - 2m + |\cyc(G(\pi))|} \|\Phi^{\prime}\|^m \\
   &= \sum_{\pi \in \Match(\sI)} K^{|\cyc(G(\pi))|} \sum_{m = 1}^{d} L^{d - m} \sum_{\substack{\rho \in \EvenPart_m(\sI) \\ \rho \preceq \pi}} N^{2d - 2m} \|\Phi^{\prime}\|^{3m - 2d} 
   \intertext{and, using that the number of $\rho$ with $m$ parts that are coarsenings of $\pi$ (which, being a matching, has $d$ parts) is equal to $|\Part_m([d])| = S_2(d, m)$, a Stirling number of the second kind, we can further rewrite upon recognizing the Touchard polynomials (Definition~\ref{def:touchard})}
   &= \left(\frac{LN^2}{\|\Phi^{\prime}\|^2}\right)^d\sum_{\pi \in \Match(\sI)} K^{|\cyc(G(\pi))|} \sum_{m = 1}^{d} S_2(d, m) \left(\frac{\|\Phi^{\prime}\|^3}{LN^2}\right)^m \\
   &= \left(\frac{LN^2}{\|\Phi^{\prime}\|^2}\right)^d T_d\left(\frac{\|\Phi^{\prime}\|^3}{LN^2}\right) \sum_{\pi \in \Match(\sI)} K^{|\cyc(G(\pi))|}  \\
   &\leq \left(\frac{LN^2}{\|\Phi^{\prime}\|^2}\right)^d \left(\frac{\|\Phi^{\prime}\|^3}{LN^2}\right)^d \exp\left(\frac{LN^2d^2}{\|\Phi^{\prime}\|^3}\right) \sum_{\pi \in \Match(\sI)} K^{|\cyc(G(\pi))|}  \tag{Corollary~\ref{cor:touchard-bound}} \\
   &= \|\Phi^{\prime}\|^d \exp\left(\frac{LN^2d^2}{\|\Phi^{\prime}\|^3}\right)  \sum_{\pi \in \Match(\sI)} K^{|\cyc(G(\pi))|}  \\
   &= (2\|\Phi^{\prime}\|)^d \exp\left(\frac{LN^2d^2}{\|\Phi^{\prime}\|^3}\right) \frac{\Gamma(d + \frac{K}{2})}{\Gamma(\frac{K}{2})} \tag{Corollary~\ref{cor:moment}}
\end{align*}

Finally, we return to the original sum we wanted to bound.
We have
\begin{align*}
    \chi^2_{\leq D}(\P_N \mid \Q_N)
    &= \sum_{d = 1}^D \frac{1}{d!} \left(\frac{\beta^2}{2N}\right)^d \E(u^{\top}\Phi^{\prime} u)^d \\
    &\leq \sum_{d = 1}^D \frac{\Gamma(d + \frac{K}{2})}{d! \cdot \Gamma(\frac{K}{2})} \left(\beta^2 \cdot \frac{\|\Phi^{\prime}\|}{N} \cdot \exp\left(\frac{LN^2 D}{\|\Phi^{\prime}\|^3}\right)\right)^d 
    \intertext{By assumption we have, as $N \to \infty$, that $L = O(1)$ and $\frac{N^2 D}{\|\Phi^{\prime}\|^3} \leq \frac{N^2 D}{\|\Phi\|^3} \to 0$, whereby for any $\eta > 0$, for sufficiently large $N$, we will have $\exp(\frac{LN^2 D}{\|\Phi^{\prime}\|^3}) \leq 1 + \eta$.
    Also, we have $\limsup_{N \to \infty} \beta^2 \cdot \frac{\|\Phi\|^{\prime}}{N} \leq \limsup_{N \to \infty} \beta^2 \cdot (\frac{\|\Phi\|}{N} + \frac{\|\Phi\|_{\ell^{\infty}}}{N}) \leq 1 - 2\epsilon$ for some $\epsilon > 0$. Let us choose $\eta$ such that $(1 - 2\epsilon)(1 + \eta) \leq 1 - \epsilon$. For sufficiently large $N$, we may bound}
    &\leq \sum_{d = 1}^D \frac{\Gamma(d + \frac{K}{2})}{d! \cdot \Gamma(\frac{K}{2})} \left(\beta^2 \cdot \frac{\|\Phi^{\prime}\|}{N} \cdot (1 + \eta)\right)^d \\
    &\leq \sum_{d = 0}^{\infty} \frac{\Gamma(d + \frac{K}{2})}{d! \cdot \Gamma(\frac{K}{2})} \left(\beta^2 \cdot \frac{\|\Phi^{\prime}\|}{N} \cdot (1 + \eta)\right)^d \\
    &= \left(1 - \beta^2 \cdot \frac{\|\Phi^{\prime}\|}{N} \cdot (1 + \eta)\right)^{-K/2} \tag{Proposition~\ref{prop:hypergeometric_sum}} \\
    &\leq \exp\left(\frac{K}{2} \cdot \beta^2 \cdot \frac{\|\Phi^{\prime}\|}{N} \cdot (1 + \eta) \cdot \frac{1 + \epsilon}{\epsilon}\right) \tag{Proposition~\ref{prop:log-bound}}
    \intertext{Finally, expanding the definition of $K$ again gives a last cancellation, leaving us with}
    &= \exp\left(\frac{\beta^2}{2} \cdot (1 + \|\Phi^{\prime}\|_{\ell^{\infty}})^2 \cdot (1 + \eta) \cdot \frac{1 + \epsilon}{\epsilon}\right) \\
    &= O(1),
\end{align*}
completing the proof, since all factors above are either constant parameters of the model, constants introduced earlier in the proof, or, for $\|\Phi^{\prime}\|_{\ell^{\infty}} \leq 2\|\Phi\|_{\ell^{\infty}}$, assumed to be $O(1)$ as $N \to \infty$.

\subsubsection{Gaussian Intuition and Symmetry Assumption}

Let us give some intuition about the part of the above manipulations where we insistently seek out the Touchard polynomials.
We have in mind that the quadratic form $u^{\top} \Phi^{\prime} u \approx u^{\top}\Phi u$ behaves as though $u$ were a standard Gaussian vector.
In this case, by rotational invariance, we could diagonalize $\Phi$ without loss of generality, and, keeping only the contribution of the largest eigenvalue $\lambda_1(\Phi)$, we would expect to have $\E (u^{\top}\Phi^{\prime} u)^d \approx \lambda_1(\Phi)^d \E h^{d}$ for $h \sim \chi^2(1)$.

When we identify that Poisson moments (Touchard polynomials) appear in our calculation, we are aiming to implement this intuition by bounding them by the corresponding moments of the $\chi^2$ distribution.
Indeed, using that $\chi^2(d) = \Gamma(d / 2, 2)$ for the family of Gamma distributions, we may embed the $\chi^2$ distributions into a continuous family, and we have observed numerically that, for all $\lambda > 0$ and $d \geq 1$,
\[ \Ex_{X \sim \Pois(\lambda)} X^d \leq \Ex_{X \sim \Gamma\left(\frac{\lambda}{2}, 2\right)} X^d, \]
both sides of which may be evaluated as polynomials in $\lambda$ for any given $d$.
We have not been able to establish this intriguing empirical inequality, but fortunately Corollary~\ref{cor:touchard-bound} gives a sufficient approximation for our purposes.

This circle of ideas is also the reason that the symmetry assumption on the distribution $\nu$ is important to our argument.
Without it, we would not be able to perform the crucial regrouping of a sum over $\rho \in \EvenPart(\sI)$ into one over $\pi \in \Match(\sI)$ with $\pi \succeq \rho$, which in turn is what eventually gives rise to the Stirling numbers and Touchard polynomials.
While this approach seems exceedingly algebraically delicate, and we suspect that the symmetry condition could be removed, we have not yet been able to find an alternative approach to the one taken here that is sufficiently precise.

\subsection{Statistical Lower Bound}

We repeat the beginning of the previous proof, but make some changes starting with \eqref{eq:general-comp-E-quad-form} above.
Using the extra assumption on moments of $\nu$, we find that this equation holds without the term $\tau^{2d - 2|\rho|}$, so we merely have
\begin{align*}
    \E(u^{\top}\Phi^{\prime} u)^d 
    &\leq \sum_{\rho \in \EvenPart(\sI)} \prod_{A \in \rho} (|A| - 1)!! \cdot f_{\rho}(\Phi^{\prime}) 
    \intertext{Using as we did in the previous proof that the combinatorial factor here is precisely the number of $\pi \in \Match(\sI)$ such that $\rho \preceq \pi$, we may bound by the dramatically simpler}
    &= \sum_{\pi \in \Match(\sI)} \sum_{\rho \preceq \pi} f_{\rho}(\Phi^{\prime}) \\
    &= \sum_{\pi \in \Match(\sI)} \tilde{f}_{\pi}(\Phi^{\prime}) \tag{Proposition~\ref{prop:partition-rearrangement}}
    \intertext{and setting $K \colonequals N\|\Phi^{\prime}\|_{\ell^{\infty}} / \lambda_1(\Phi^{\prime})$, we have}
    &\leq \lambda_1(\Phi^{\prime})^d \sum_{\pi \in \Match(\sI)} K^{|\cyc(G(\pi))|} \cdot \tag{Proposition~\ref{prop:cycles}} \\
    &= (2\lambda_1(\Phi^{\prime}))^d \frac{\Gamma(d + \frac{K}{2})}{\Gamma(\frac{K}{2})}, \tag{Corollary~\ref{cor:moment}}
\end{align*}
reaching a simplified version of our final bound on this expectation from the previous proof.

\begin{remark}
    The first bound above would hold with equality if the vector $u$ consisted of i.i.d.\ entries distributed as $\sN(0, 1)$, as may be verified by applying Wick's formula.
    This observation is the motivation for our approach to these calculations and to the form of our assumption on the moments of $\nu$ for this proof.
    Note, however, that $u = x^1 \odot x^2$ for two draws of $x^i$ from the distribution of our prior $\sP_N$, so this intuition is \emph{not} saying that the above calculation behaves as it would for $\sP_N = \sN(0, I_N)$ (to which case this proof does not apply).
    Rather, the idea is that a situation like $\sP_N = \Unif(\{\pm 1\}^N)$, in which case the law of $u$ is again $\Unif(\{\pm 1\}^N)$, should resemble the case where $u \sim \sN(0, I_N)$.
\end{remark}

Repeating the same manipulations from the previous proof, we find, letting $\epsilon > 0$ be such that $\limsup_{N \to \infty} \beta^2 \cdot \frac{\lambda_1(\Phi^{\prime})}{N} \leq 1 - \epsilon$,
\begin{align*}
    \chi^2(\P_N \mid \Q_N)
    &\leq \left(1 - \beta^2 \cdot \frac{\lambda_1(\Phi^{\prime})}{N}\right)^{-K / 2} \\
    &\leq \exp\left(\frac{K}{2} \cdot \beta^2 \cdot \frac{\lambda_1(\Phi^{\prime})}{N} \cdot \frac{1 + \epsilon}{\epsilon}\right) \\
    &= \exp\left(\frac{\|\Phi^{\prime}\|_{\ell^{\infty}}}{2} \cdot \beta^2 \cdot \frac{1 + \epsilon}{\epsilon}\right) \\
    &= O(1),
\end{align*}
completing the proof.

\subsection{Growth of Low-Degree \texorpdfstring{$\chi^2$-}{Chi-Squared }Divergence: Proof of Theorem~\ref{thm:general-low-deg-div}}
\label{sec:pf:thm:general-low-deg-div}

We follow the same general approach as in the proof of Theorem~\ref{thm:block-low-deg-div} in Section~\ref{sec:pf:thm:block-low-deg-div}.
As we argued there, it suffices to show that
\[ \E \exp^{\leq D}(R^{(0)}) = \omega(1), \]
for
\[ R^{(0)} = \frac{\beta^2}{2N} u^{\top}\Phi u, \]
where $u = x^1 \odot x^2$ for $x^1, x^2 \sim \sP_N$ two independent draws from the signal distribution.

The main difference is that, in Theorem~\ref{thm:block-low-deg-div}, we could use the block structure of $\Phi$ to group this sum into a finite number of sums to each of which the central limit theorem would apply.
Here, we cannot do that, but we take a similar approach with a generalized central limit theorem.
Let $\mu_1 \geq \cdots \geq \mu_N$ be the ordered eigenvalues of $\Phi^{(N)}$, and $v_1, \dots, v_N$ be the associated unit eigenvectors.
We may then rewrite
\[ R^{(0)} 
= \frac{\beta^2}{2N} \sum_{i = 1}^N \mu_i \langle v_i, u\rangle^2 \]

We now claim that we have the convergence in distribution
\[ (\langle v_1, u \rangle, \dots, \langle v_K, u \rangle) \dto \sN(0, I_K), \]
for any $K$ fixed as $N \to \infty$.
Recall that we assume either that $\Phi^{(N)} \succeq 0$, in which case we will use this with $K = 1$, or that $\Phi^{(N)}$ has bounded rank, in which case we will take $K$ to be that bound.
By Theorem~\ref{thm:cramer-wold}, it suffices to show that, for any $c \in \R^K$ with $\|c\| = 1$, we have
\begin{equation}
\label{eq:clt-cond}
\left\langle \sum_{i = 1}^K c_i v_i, u \right\rangle \dto \sN(0, 1). 
\end{equation}
Recall that we have assumed that
\[ \lim_{N \to \infty} \max_{i = 1}^{N} \|v_i\|_{\infty} = 0. \]
So, we also have
\[ \lim_{N \to \infty} \left\|\sum_{i = 1}^{K} c_i v_i\right\|_{\infty} = 0. \]
Thus, since the entries of $u$ are i.i.d., bounded, and have mean zero and variance 1, the convergence in \eqref{eq:clt-cond} follows by the Lindeberg central limit theorem.
Let us write $(g_1, \dots, g_K) \sim \sN(0, I_K)$ for this distributional limit.

We also have in general that
\begin{align*}
    \E \exp^{\leq D}(R^{(0)})
    &= \sum_{d = 0}^D \frac{1}{d!} \E R^{(0)^d}
    \intertext{and every term here is non-negative. Further, since $D = D(N) \to \infty$ by assumption, for any given $d_0 \geq 1$ we will eventually have $D \geq d_0$ for sufficiently large $N$, and so we can bound}
    &\geq \frac{1}{d_0!} \E R^{(0)^{d_0}}.
\end{align*}

We now consider our two possible assumptions on $\Phi^{(N)}$ separately.
If $\Phi^{(N)} \succeq 0$, then $R^{(0)} \geq 0$ almost surely, and may bound by just the term of the top eigenvalue $\mu_1$,
\[ R^{(0)} \geq \frac{\beta^2}{2N}\mu_1 \langle v_1, u \rangle^2 \dto \frac{\beta^2 \what{\mu}_1}{2} g_1^2. \]
As in the proof of Theorem~\ref{thm:block-low-deg-div}, a standard approximation argument implies that this convergence also holds for moments, and thus
\begin{align*}
    \E \exp^{\leq D}(R^{(0)})
    &\geq \frac{1}{d_0!} \E R^{(0)^{d_0}} \\
    &\geq \frac{1}{d_0!} \E \left(\frac{\beta^2}{2N}\mu_1 \langle v_1, u \rangle^2\right)^{d_0} \\
    &\to \frac{1}{d_0!} \E \left(\frac{\beta^2\what{\mu}_1}{2} g_1^2\right)^{d_0} \\
    &= \frac{(2d_0 - 1)!!}{(2d_0)!!} (\beta^2\what{\mu}_1)^{d_0}.
\end{align*}
Stirling-type approximations imply that the first term is polynomial in $d$, while by assumption $\beta^2 \what{\mu}_1 > 1$, so taking $d_0 \to \infty$ gives a diverging lower bound.

Now, suppose we instead have the second possible condition on $\Phi^{(N)}$, that $\rank(\Phi^{(N)}) \leq K$ and $\lambda_i(\Phi^{(N)}) / N \to \what{\mu}_i$ for each $i = 1, \dots, K$.
Then, we have that $R^{(0)}$ itself converges in distribution:
\[ R^{(0)} = \frac{\beta^2}{2N}\sum_{i = 1}^K\mu_i \langle v_i, u \rangle^2 \dto \frac{\beta^2}{2}\sum_{i = 1}^K \what{\mu}_i g_i^2. \]
By the same argument as above, we have
\begin{align*}
    \liminf_{N \to \infty} \exp^{\leq D}(R^{(0)})
    &\geq \frac{1}{d_0!} \E\left(\frac{\beta^2}{2}\sum_{i = 1}^K \what{\mu}_i g_i^2\right)^{d_0} \\
    &= \frac{1}{d_0!} \E\left(\frac{\beta^2}{2}g^{\top} D g\right)^{d_0}
    \intertext{for $D = \Diag(\what{\mu}_1, \dots, \what{\mu}_K)$.
    Note that, for any $\ell \geq 1$, $\Tr(D^{\ell}) = \lim_{N \to \infty} \Tr(\Phi^{(N)} / N)^k \geq 0$ since all entries of $\Phi^{(N)}$ are non-negative.
    Thus, we may apply Proposition~\ref{prop:baopaper}, which gives}
    &\geq \frac{1}{d_0!} \left(\frac{\beta^2}{2}\right)^{d_0} 2^{d_0 - 1} (d_0 - 1)! \Tr(D^{d_0}) \intertext{and, provided we take $d_0$ even,}
    &\geq \frac{1}{2d_0} (\beta^2 \what{\mu}_1)^{d_0},
\end{align*}
and the proof concludes after taking $d_0 \to \infty$ as in the first case.

\subsection{Numerical Experiment}
\label{sec:numerical}

While our results suggest that the computational threshold for general variance profiles of noise is given by a straightforward extension of the formula for the threshold for block-structured profiles, they do not say anything about whether the spectral algorithm studied by \cite{pak2024optimal,MKK-2024-OptimalPCABlockSpikedMatrix,BCSVH-2024-MatrixConcentrationFreeProbability2} for block-structured profiles should still be effective for general profiles.
Still, the formal similarity between the thresholds does suggest the intriguing possibility that the spectral algorithm using $\sH(Y)$ might remain optimal for general variance profiles.
Here we give some numerical evidence that, at least, that spectral algorithm is superior to two other natural choices.
The theoretical discussion in this section is all at a heuristic level and is meant merely to justify the setup of our experiment and perhaps to suggest possible explanations of the results.

\subsubsection{Alternative Algorithms and Predicted BBP Thresholds}

The two alternative algorithms we consider (also discussed by \cite{GKKZ-2025-InhomogeneousSpikedWigner,MKK-2024-OptimalPCABlockSpikedMatrix}) are as follows.
The first is just to consider the top eigenpair of $\what{Y} = Y / \sqrt{n}$ itself.
We call this the \emph{direct} spectral algorithm, and denote its BBP threshold (i.e., the value of $\beta$ at which the transition from Theorems~\ref{thm:bbp} and \ref{thm:bbp-inhom} occurs) by $\beta_*^{\dir}$.
The second is to consider the different and more directly intuitive modification of $Y$,
\[ \sG(Y) \colonequals \frac{1}{\sqrt{N}} \Delta^{\odot -1/2} \odot Y, \]
which divides each entry of $Y$ by the standard deviation of the noise applied to that entry.
We call this the \emph{whitening} spectral algorithm, because it makes the ``noise term'' of $\sG(Y)$ i.i.d.\ $\sN(0, 1)$, in exchange for modifying the ``signal term'' in a possibly complicated way.
We call the BBP threshold of this algorithm $\beta_*^{\wh}$.
Finally, we call the \emph{linearized AMP} or \emph{LinAMP} spectral algorithm the one using $\sH(Y)$, and call its BBP threshold $\beta_*^{\LinAMP}$. 

Let us assume for the sake of simplicity that the signal $x$ in our model has law $x \sim \Unif(\{\pm 1\}^n)$, i.e., we take the entrywise distribution $\nu = \Unif(\{\pm 1\})$ in the notation of the main results.
We now describe the values we expect each of these three thresholds to take.
To avoid needing to carefully prescribe what sequences of variance profiles ``converge'' in a suitable sense, we keep the discussion entirely heuristic and make non-asymptotic predictions about these thresholds.

First, from our results in this paper, we are led to hope for the LinAMP spectral algorithm that
\begin{equation}
\label{eq:bbp-thresh-opt}
\beta_*^{\LinAMP} \approx \sqrt{\frac{N}{\lambda_1(\Phi)}} = \sqrt{\frac{N}{\lambda_1(\Delta^{\odot -1})}}. 
\end{equation}
Second, for the whitening spectral algorithm, note that
\[ \sG(Y) = \frac{1}{\sqrt{N}} \beta \Delta^{\odot -1/2} \odot xx^{\top} + W^{(0)} = \frac{1}{\sqrt{N}} \beta D_x  \Delta^{\odot -1/2} D_x + W^{(0)} \]
where the entries of $W^{(0)}$ on and above the diagonal are i.i.d.\ with distribution $\sN(0, 1)$.
This is a spiked matrix model with homogeneous noise but where the signal part $\frac{1}{\sqrt{N}} \beta D_x \Delta^{\odot -1/2} D_x$ does not necessarily have low rank.
Conveniently, by our assumption on $x$ we have that $D_x$ is an orthogonal matrix, so the eigenvalues of this signal part up to rescaling are those of $\Delta^{\odot -1/2}$.
Provided this matrix is approximately low rank, per the analysis of, for example, \cite{CDMF-2009-DeformedWigner}, we expect the threshold
\begin{equation}
\label{eq:bbp-thresh-wh}
\beta_*^{\wh} \approx \frac{N}{\lambda_1(\Delta^{\odot -1/2})}.
\end{equation}
Note that \cite{GKKZ-2025-InhomogeneousSpikedWigner} show that, if we take the above as definitions of $\beta_*^{\wh}$ and $\beta_*^{\LinAMP}$, then for any variance profile $\Delta$ we have $\beta_*^{\wh} \geq \beta_*^{\LinAMP}$, with equality if and only if $\Delta$ is constant.

Finally, we propose that a reasonable proxy for the BBP threshold of the original inhomogeneous model $\what{Y}$ is to apply the theory for spiked matrix models with orthogonally invariant noise, in particular the results of \cite{BGN-2011-PerturbationsRandomMatrices}.
(The more involved analysis of \cite{bigot2021freeness} also applies to this question, but we take this more heuristic approach to lighten the numerical computations required.)
While our matrix $W$ is not orthogonally invariant, if for instance we instead had $x$ drawn uniformly from the sphere of radius $\sqrt{N}$, then we could consider $Q \what{Y} Q^{\top}$ instead of $\what{Y}$ for $Q$ a Haar-distributed orthogonal matrix and use that $QWQ^{\top}$ is then orthogonally invariant.
In any case, writing $\what{W} \colonequals \what{W} / \sqrt{N}$, the results of \cite{BGN-2011-PerturbationsRandomMatrices} describe the BBP threshold for such $\what{Y}$ in terms of the limiting empirical eigenvalue distribution of $\what{W}$ over a suitably converging sequence.
Showing that certain sequences of matrices with variance profiles have these distributions converging is itself non-trivial, so let us just suppose we believe this distribution to be some compactly supported $\gamma$.
If $b$ is the right edge of the support of $\gamma$, then \cite{BGN-2011-PerturbationsRandomMatrices} yields the prediction
\begin{equation}
\label{eq:bbp-thresh-dir}
\beta_*^{\dir} \approx \frac{1}{\lim_{z \to b^+} G(z)},
\end{equation}
where $G(z)$ is the \emph{Cauchy transform} of $\gamma$, given by
\[ G(z) = \Ex_{x \sim \gamma} \left[\frac{1}{z - x}\right]. \]

For our purposes here, given a specific variance profile matrix $\Delta$, we would like to compute these three thresholds to compare with numerical results.
While for the LinAMP and whitening algorithms \eqref{eq:bbp-thresh-opt} and \eqref{eq:bbp-thresh-wh} respectively give simple such formulas, it is unclear how to compute with \eqref{eq:bbp-thresh-dir}, since we do not have access to the measure $\gamma$ or the number $b$.
Given a concrete $\Delta$ and $W$ sampled from $\Delta$, we will simply approximate
\[ b \approx \lambda_1(\what{W}). \]
Then, following the heuristic derivation discussed, for instance, on the first pages of \cite{ajanki2019quadratic}, we use that $G(z)$ may be estimated by first solving the recursion
\[ -\frac{1}{m_i} = z + \frac{1}{N} \sum_{j = 1}^N \Delta_{ij} m_j = z + \frac{1}{N}(\Delta m)_i \]
for a vector $m = (m_1, \dots, m_N) \in \R^N$, and then taking the approximation
\[ G(z) \approx -\frac{1}{N} \sum_{i = 1}^N m_i. \]
See the citation above for a justification of this procedure involving approximating the resolvent matrix of $\what{W}$.
We then estimate $\beta_*^{\dir}$ by performing this procedure (concretely, we obtain $m$ by fixed point iterations of the above equations) with $z$ equal to our estimate of $b$, $z \colonequals \lambda_1(\what{W})$.

\subsubsection{Empirical Results for a Continuous Variance Profile}

Now let us propose a specific non-block-structured variance profile and describe the above predicted thresholds for the three spectral algorithms.
In general, a widely studied type of variance profile in random matrix theory (though, to the best of our knowledge, not yet in the context of spiked matrix models) is a profile of the form
\begin{equation}
\label{eq:smooth-profile}
\Delta_{ij} \colonequals f\left(\frac{i}{N}, \frac{j}{N}\right)
\end{equation}
for a continuous symmetric function $f: [0, 1]^2 \to \R_{\geq 0}$ (see the discussion in Section~\ref{sec:related}).
If $f$ is also bounded away from zero, then such a variance profile will satisfy the condition \eqref{eq:uniform-profile-bounds} discussed earlier, and our Theorem~\ref{theorem:general-bound} will apply (provided that Assumption~\ref{assumptionA} holds, which also may be shown to hold for this family of choices).
Such variance profiles are also the context in which the description of $\gamma$ mentioned above has been shown to hold; they are convenient for the above description since $\frac{1}{N}\sum_{j = 1}^N \Delta_{ij} m_j$ then becomes a Riemann sum if $m_j = m(j / N)$ is associated to another function.

\begin{figure}
    \begin{center}
    \includegraphics[scale=0.6]{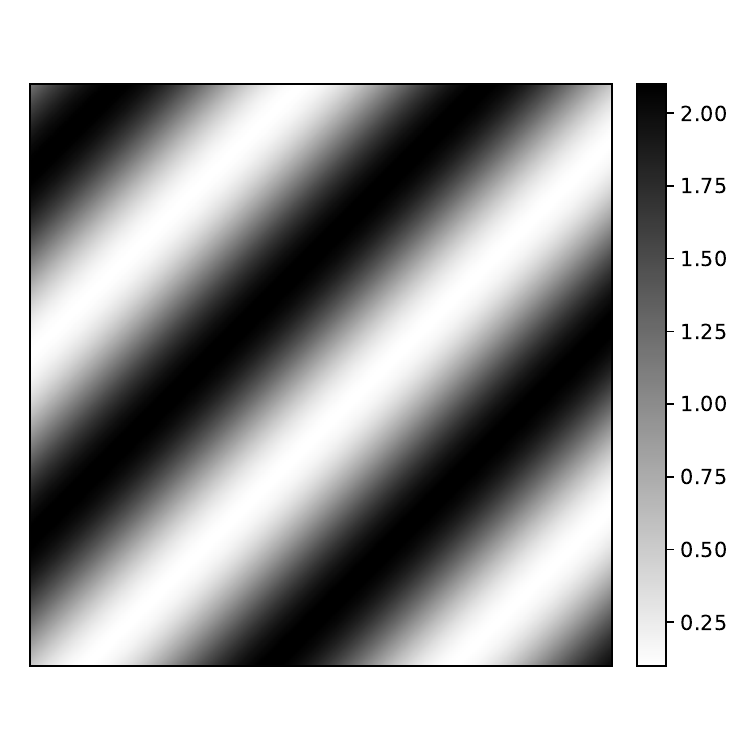}
    \end{center}
    \vspace{-0.8cm}
    \caption{An illustration of the variance profile matrix $\Delta$, plotted as a heatmap of an $N \times N$ matrix with $N = \num{12000}$, defined by \eqref{eq:smooth-profile} with the choice of the function $f$ given in \eqref{eq:profile-f}. For such large $N$, this heatmap is visually indistinguishable from a plot of the continuous function $f$ itself on the domain $[0, 1]^2$.}
    \label{fig:profile}
\end{figure}

We will consider such a variance profile for the function
\begin{equation}
\label{eq:profile-f}
f(x, y) = \frac{11}{10} + \sin(10(x + y)).
\end{equation}
Note that we indeed have the uniform bounds $1/10 \leq f(x, y) \leq 21 / 10$.
We illustrate the variance profile associated to such an $f$ in Figure~\ref{fig:profile}.
Performing the above calculations for this variance profile evaluated with $N = \num{12 000}$ gives
\begin{align*}
    \beta_*^{\LinAMP} &\approx \num{0.653}, \\
    \beta_*^{\wh} &\approx \num{0.759}, \\
    \beta_*^{\dir} &\approx \num{1.158}.
\end{align*}
In Figure~\ref{fig:histograms}, we show eigenvalue distributions of the matrices constructed by each algorithm for a sequence of values
\[ \beta_1 < \beta_*^{\LinAMP} < \beta_2 < \beta_*^{\wh} < \beta_3 < \beta_*^{\dir} < \beta_4. \]
Specifically, we take
\begin{align*}
    \beta_1 &= \num{0.65}, \\
    \beta_2 &= \num{0.75}, \\
    \beta_3 &= \num{0.85}, \\
    \beta_4 &= \num{1.25}.
\end{align*}
We observe that, as expected, for $\beta = \beta_1$ none of the matrices has an outlier eigenvalue, for $\beta = \beta_2$ only the LinAMP algorithm's matrix has an outlier eigenvalue, for $\beta = \beta_3$ only the LinAMP and 
whitening algorithm's matrices have outlier eigenvalues, and for $\beta = \beta_4$ all three matrices have outlier eigenvalues.

\begin{remark}
    Note that, as is visible in the figure, the bulk distribution of eigenvalues for the direct and whitening algorithms' matrices $\what{Y}$ and $\sG(Y)$ does not depend on $\beta$, while the bulk distribution for the LinAMP algorithm's matrix $\sH(Y)$ does.
    This is just because the construction of $\sH(Y)$ presumes access to $\beta$ and includes $\beta$ in its very definition; see \eqref{eq:tildeY}.
\end{remark}

We emphasize the finding that, at least on this particular variance profile, the LinAMP algorithm is indeed superior to the other two, and its performance is consistent with the predicted threshold $\beta_*^{\LinAMP}$.
We leave it as an intriguing open question to determine whether the actual BBP threshold for $\sH(Y)$ coincides with our $\beta_*$ parameter more generally, to understand whether a spectral algorithm using $\sH(Y)$ is indeed optimal for many more variance profiles (and if so, under what conditions on variance profiles) or whether this matrix construction needs to be modified further in the non-block-structured case.

\begin{figure}
\begin{center}
    \begin{tabular}{cccc}
    & Direct & Whitening & Linearized AMP \\[0.65em]
    \rotatebox{90}{\hspace{1.4cm}$\beta = \beta_1$} & \includegraphics[scale=0.45]{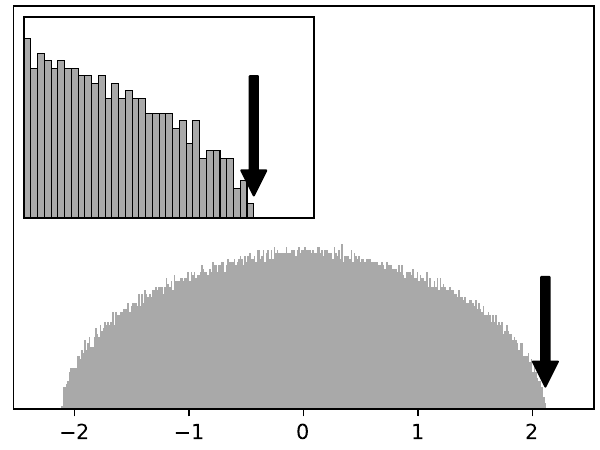} & \includegraphics[scale=0.45]{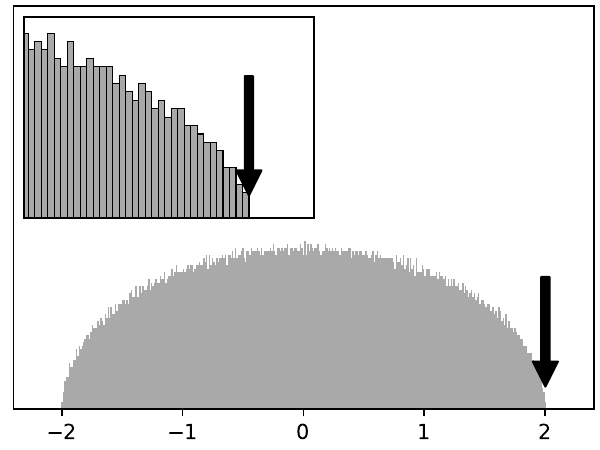} & \includegraphics[scale=0.45]{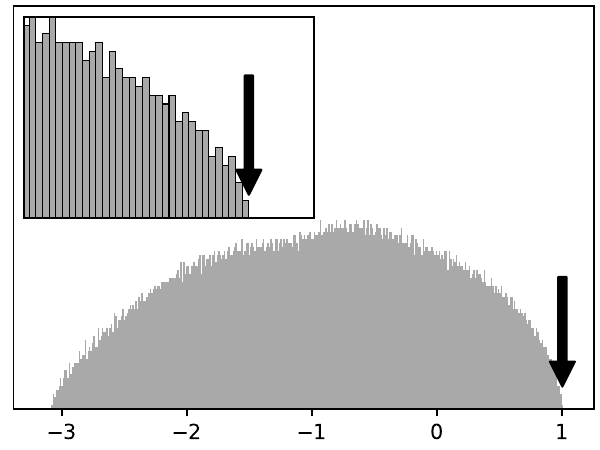} \\
    \rotatebox{90}{\hspace{1.4cm}$\beta = \beta_2$} & \includegraphics[scale=0.45]{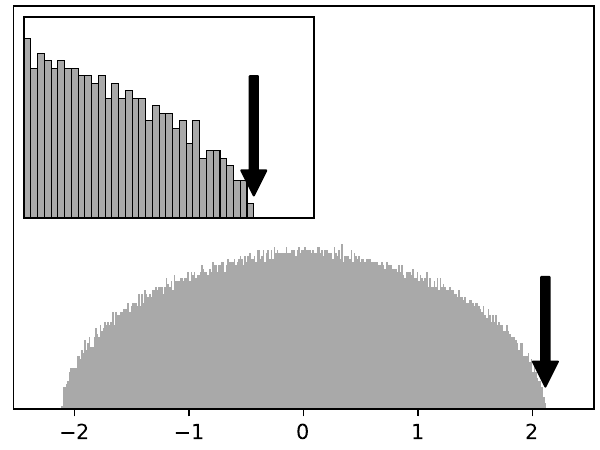} & \includegraphics[scale=0.45]{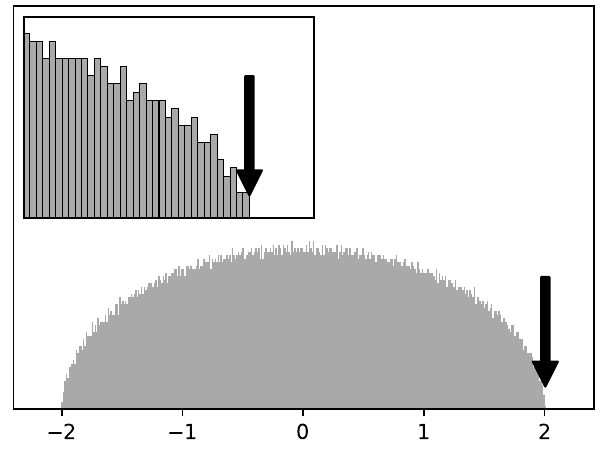} & \includegraphics[scale=0.45]{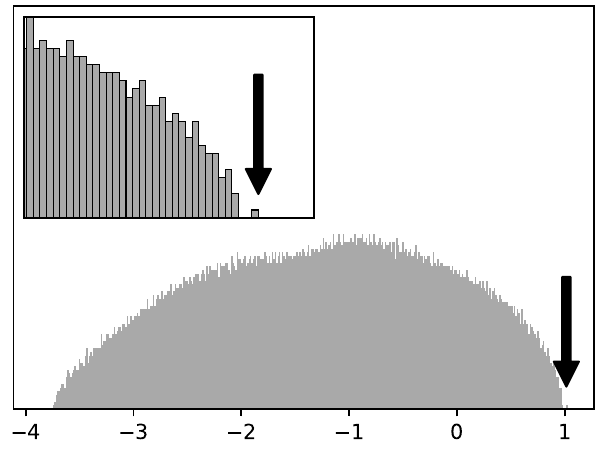} \\
    \rotatebox{90}{\hspace{1.4cm}$\beta = \beta_3$} & \includegraphics[scale=0.45]{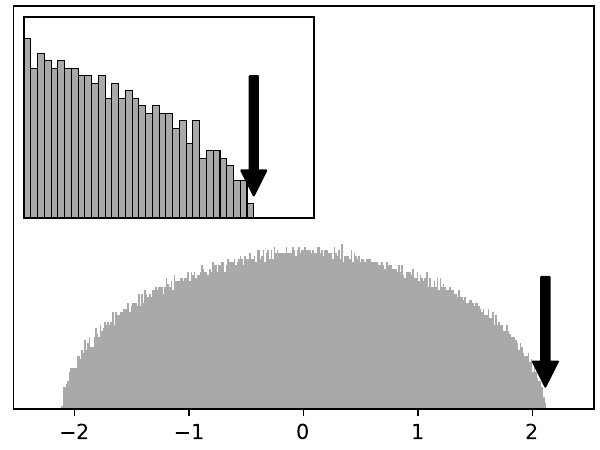} & \includegraphics[scale=0.45]{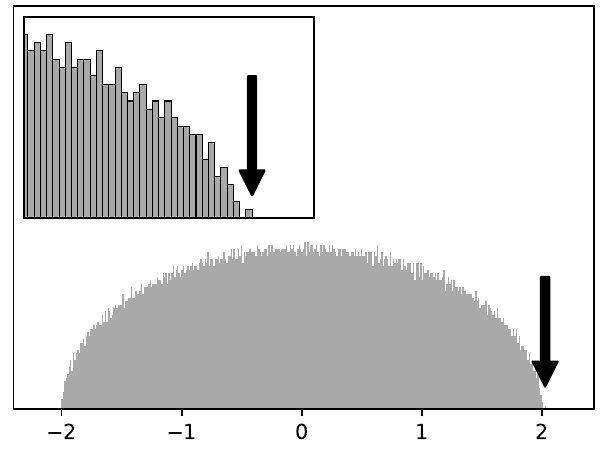} & \includegraphics[scale=0.45]{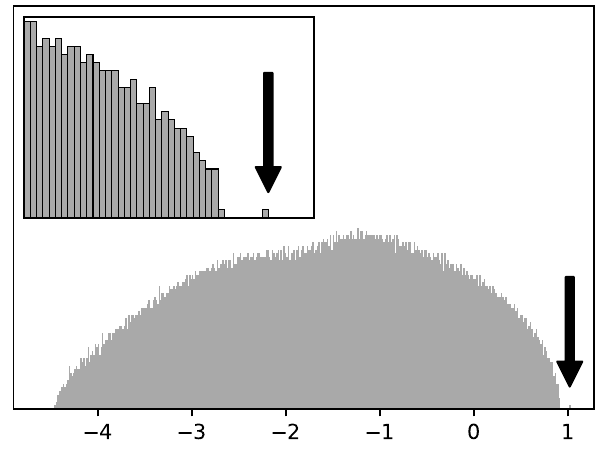}  \\
    \rotatebox{90}{\hspace{1.4cm}$\beta = \beta_4$} & \includegraphics[scale=0.45]{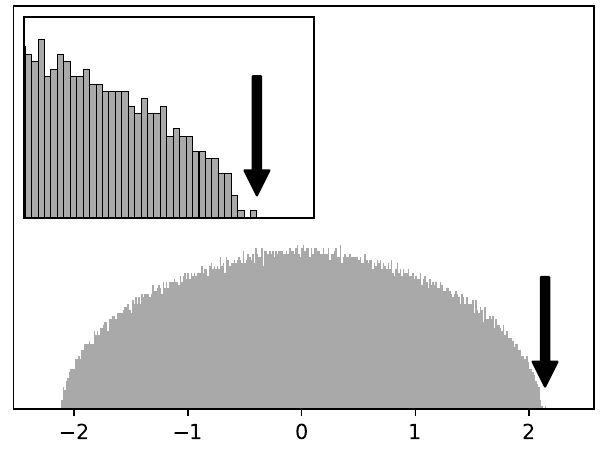} & \includegraphics[scale=0.45]{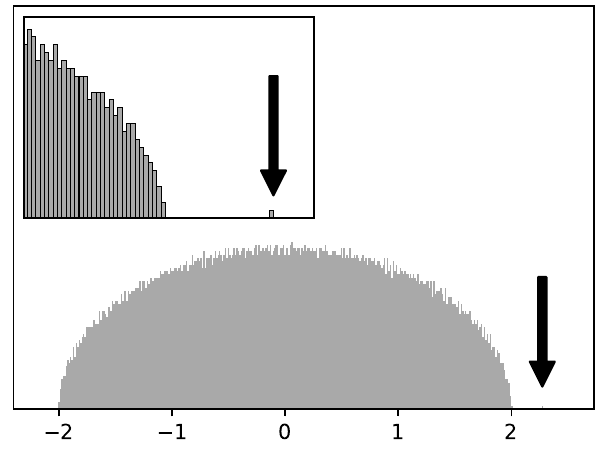} & \includegraphics[scale=0.45]{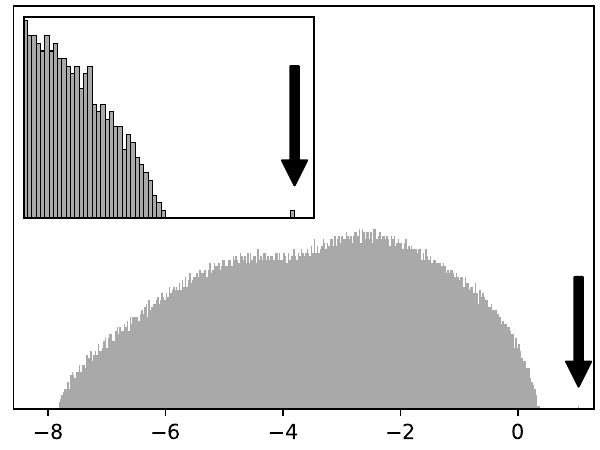}
    \end{tabular}
    \end{center}
    \caption{We compare the direct, whitening, and linearized AMP (LinAMP) spectral algorithms discussed in Section~\ref{sec:numerical} for several signal strengths $\beta_1 < \beta_2 < \beta_3 < \beta_4$ and $N = \num{12000}$, plotting the eigenvalue distribution of the matrix each constructs with the maximum eigenvalue indicated with an arrow. As predicted, the LinAMP algorithm requires the smallest threshold signal strength to observe an outlier eigenvalue, the whitening algorithm the next-smallest strength, and the direct algorithm the largest threshold strength.}
        \label{fig:histograms}
\end{figure}

\clearpage

\bibliographystyle{alpha}
\bibliography{bibdata}

\end{document}